\documentclass[a4paper,12pt,oneside]{article}

\usepackage[hidelinks]{hyperref}
\usepackage{amsmath}
\usepackage{amssymb}
\usepackage{amsthm}
\usepackage{dsfont}
\usepackage{graphicx}
\usepackage[lmargin=25mm,rmargin=25mm,top=25mm,bottom=30mm,paperwidth=210mm,paperheight=11in]{geometry}
\usepackage{mathrsfs}
\usepackage{pifont}
\usepackage{tikz}

\usepackage[utf8]{inputenc}

\date{}

\newtheorem{theorem}{Theorem}[section]
\newtheorem*{theorem*}{Theorem}
\newtheorem{lemma}{Lemma}[section]
\newtheorem{proposition}[lemma]{Proposition}
\newtheorem{corollary}[lemma]{Corollary}
\theoremstyle{definition}
\newtheorem*{definition*}{Definition}

\newcommand{\E}{\mathbb{E}}
\newcommand{\N}{\mathbb{N}}

\renewcommand{\P}{\mathbb{P}}
\newcommand{\R}{\mathbb{R}}
\newcommand{\Z}{\mathbb{Z}}

\newcommand{\1}{\mathds{1}}

\newcommand{\C}{\mathcal{C}}
\newcommand{\D}{\mathcal{D}}
\renewcommand{\o}{{\boldsymbol{0}}}
\newcommand{\LL}{\mathbb{L}}
\newcommand{\F}{\mathcal{F}}
\newcommand{\oeta}{\overline{\eta}}
\newcommand{\tM}{\tilde{M}}
\newcommand{\teta}{\tilde{\eta}}

\DeclareMathOperator*{\diam}{diam}
\newcommand{\dd}{\mathrm{d}}
\renewcommand{\epsilon}{\varepsilon}

\newcommand{\law}{\mathscr{L}}
\newcommand{\precc}{\preccurlyeq}

\renewcommand{\leq}{\leqslant}
\renewcommand{\le}{\leqslant}
\renewcommand{\geq}{\geqslant}
\renewcommand{\ge}{\geqslant}
\renewcommand{\subset}{\subseteq}

\usepackage{pgfplots}
\usetikzlibrary{calc}
\usetikzlibrary{patterns}
\definecolor{myblue}{rgb}{0,0.5,1}
\definecolor{mypink}{rgb}{1,0,1}
\definecolor{mygreen}{rgb}{0,0.6,0}
 \pgfdeclarepatternformonly{north west spaced lines}{\pgfqpoint{-1pt}{-1pt}}{\pgfqpoint{10pt}{10pt}}{\pgfqpoint{9pt}{9pt}}%
{
  \pgfsetlinewidth{0.4pt}
  \pgfpathmoveto{\pgfqpoint{0pt}{10pt}}
  \pgfpathlineto{\pgfqpoint{10.1pt}{-0.1pt}}
  \pgfusepath{stroke}
}

\pgfplotsset{compat=1.13}

\newcommand{\ontop}[2]{\genfrac{}{}{0pt}{}{#1}{#2}}

\ifx \Leo \undefined
\else
\fi

\begin{document}

\title{Scaling limit of subcritical contact process}

\author{Aurelia Deshayes and Leonardo T. Rolla
\\
\small
Argentinian National Research Council at the University of Buenos Aires
\\
\small
NYU-ECNU Institute of Mathematical Sciences at NYU Shanghai}

\maketitle

\begin{abstract}
In this paper we study the subcritical contact process on $\Z^d$ for large times, starting with all sites infected.
The configuration is described in terms of the macroscopic locations of infected regions in space and the relative positions of infected sites in each such region.
\end{abstract}

This preprint has the same numbering of sections, equations, figures and theorems as the the
published article
``\emph{Stochastic Process. Appl. 127 (2017): 2630--2649.}''

\section{Introduction}

We consider the classical contact process on $\Z^d$, briefly described as follows.
The state at time $t$ is a subset $\eta_t \subseteq \Z^d$, or equivalently an element $\eta_t \in \{0,1\}^{\Z^d}$.
Each infected site $x$ (\textsl{i.e.}\ $x\in\eta_t$) heals spontaneously (\textsl{i.e.}\ is removed from $\eta_t$) at rate $1$.
Each healthy site $x$ (\textsl{i.e.}\ $x\not\in\eta_t$) gets infected (\textsl{i.e.}\ is added to $\eta_t$) at rate given by the number of its nearest neighbors $y$ that are infected at time $t$, multiplied by $\lambda>0$.
The number $\lambda$ is the only parameter of this time evolution.
For $A\subset \Z^d$, we denote by $(\eta_t^A)_{t\ge 0}$ the process starting from $\eta_0=A$.
When $A$ is random and has distribution $\mu$, we denote the process by $\eta_t^\mu$.
When $A=\{x\}$ we write $\eta_t^x$ and when $A=\Z^d$ we may omit the superindex.

The contact process is one of the simplest interacting particle systems that exhibit a phase transition.
There exists a non-trivial critical value $0<\lambda_c<\infty$ such that the probability that an infection starting from a single site propagates indefinitely is positive when $\lambda>\lambda_c$ and zero when $\lambda < \lambda_c$.
See~\cite{IPS,lig99} for background on this and related models.

\medskip

In this paper we study the subcritical phase.
Our goal is to describe the configuration $\eta_t$ starting from $\eta_0=\Z^d$, for large values of $t$.
We describe it in terms of the macroscopic locations of infected regions in space and the relative positions of infected sites in each such region.
Hereafter we assume that $0<\lambda<\lambda_c$ is fixed.

\medskip

To see something in this phase, we have to start with an infinite initial configuration, and search for the infected sites.
In one dimension it is common to start from a configuration that is infinite only to the left, and consider the ``contact process seen from the rightmost point.''
In~\cite{ASS90} it is shown that the subcritical contact process seen from the rightmost point has no invariant measures.
In~\cite{andjelrolla} it is shown that the process nonetheless converges in distribution.

\medskip

In higher dimensions, the idea of ``seen from a specific infected site'' can be replaced by considering the \emph{contact process modulo translations}, at least when finite configurations are being considered.
We say that two non-empty finite configurations $\eta$ and $\eta' \subset \Z^d$ are equivalent if they are translations of each other.
Let $\Lambda$ denote the quotient space resulting from this equivalence, and let $\langle \eta \rangle$ denote the projection of a finite configuration $\eta$ onto $\Lambda \cup \{\emptyset\}$.
The \emph{contact process modulo translations} is the process $(\zeta_t)_{t\ge 0}$ given by $\zeta_t=\langle \eta_t \rangle \in \Lambda\cup\{\emptyset\}$.
Since the evolution rules of $(\eta_t)_{t\ge 0}$ are translation-invariant, the process $(\zeta_t)_{t\ge 0}$ is a homogeneous Markov process.
The set $\Lambda$ is an irreducible class and the absorbing state $\emptyset$ is reached almost-surely.

We say that a probability distribution $\mu$ on $\Lambda$ is a \emph{quasi-stationary distribution}, or simply QSD, if, for every $t>0$, $\law(\zeta_t^{\mu} | \zeta^\mu_t \ne \emptyset)=\mu$.

\begin{proposition}
[\cite{sturmswart,andjelrolla}]
\label{prop:existencenu}
The subcritical contact process modulo translations on $\Z^d$ has a QSD $\nu$ on $\Lambda$ (which depends on $d$ and $\lambda$) such that the Yaglom limit
\begin{equation}
\label{eq:yaglom}
\law \left( \zeta_t^{A} \middle| \zeta_t^{A} \ne \emptyset \right) \to\nu \text{ as }t\to\infty
\end{equation}
holds for any finite non-empty initial configuration $A$.
\end{proposition}

The absorption time $\tau^{\mu}=\inf\{t\ge 0 : \zeta_t^{\mu}=\emptyset\}$ starting from a QSD $\mu$ is exponentially distributed (by the Markov Property).
Let $\alpha$ denote the absorption rate of the limiting distribution $\nu$, that is,
\begin{equation}
\label{eq:decexpo}
\P(\zeta^\nu_t \ne \emptyset)=e^{-\alpha t}.
\end{equation}

Using Proposition~\ref{prop:existencenu} and controlling the statistical effect of picking the rightmost infected site, it is shown in~\cite{andjelrolla} that, for any infinite initial configuration $A\subseteq -\N$, the subcritical contact process seen from the rightmost point converges in distribution to $\nu$.

\medskip

In this paper we provide a more detailed description of the collection of infected regions, in any dimension.
Let $R_t \in \N$ be such that
\begin{equation}
\nonumber
1 \ll R_t \ll e^{\alpha t / d}.
\end{equation}
At time $t>0$, we group the sites in the configuration $\eta_t$ by connecting those which are at distance less than $R_t$.
Let $\C_t$ denote the set of connected components.
Each finite component $D\in\C_t$ can be identified by a pair $(x,\zeta)$, where the choice of $x$ among the sites of $D$ is arbitrary (e.g. lexicographical order for simplicity) and $\zeta=\langle D \rangle \in \Lambda$ describes the relative positions of such sites.
\begin{theorem}
\label{thm:convergence}
For the subcritical contact process on $\Z^d$, the marked point process consisting of macroscopic location and relative position of infected sites satisfies
\label{main}
\[
\sum_{(x,\zeta)\in\C_t} \delta_{ \displaystyle ( e^{ -{\alpha} t /d }x,\zeta ) } \longrightarrow \mathrm{Poisson}(\rho \, \dd x \times \nu)
\text{ in distribution as }
t \to \infty
,
\]
where $\dd x$ is the Lebesgue measure, $\nu$ is given by~(\ref{eq:yaglom}) and $0<\rho<\infty$ is given by~(\ref{eq:rho}) below.
\end{theorem}

The difficulty in studying the infected regions comes from the interplay between two factors: the lack of finite-range dependence of the contact process, and the fact that by searching for infected regions one finds pieces of space where the process is not typical.
To handle these effects simultaneously, we use a classification of ``good points'' and the construction of a ``minimal path,'' described in \textsection\ref{sec:goodpoint} and \textsection\ref{sec:minimalpath}.
Both ideas were introduced in~\cite{andjelrolla}, however they relied on the planar topology of the one-dimensional nearest-neighbor model, and the extension to a more general setting was not straightforward.

The arguments presented here provide an analogous scaling limit for the discrete-time analogue called oriented percolation.

\bigskip

We conclude with a brief discussion about the method of proof.

Most of the work is devoted to extend the convergence stated in Proposition~\ref{prop:existencenu} to a scenario where $\eta_0=\Z^d$ instead of a fixed finite set $A$, and the conditioning is on the existence of infected sites within a large finite box.
Denote $B_r=[-r,r]^d \cap \Z^d$.
\begin{proposition}
\label{prop:cvtonu}
Let $(\eta_t)_{t\ge0}$ be the subcritical contact process on $\Z^d$.
Then
\begin{equation}
\nonumber
\law\left(\langle\eta_t\cap B_{R_t}\rangle \,\middle|\, \big. \eta_t\cap B_{R_t}\neq \emptyset\right)\xrightarrow[t\to\infty]{TV}\nu
\end{equation}
under the assumption $t^{d+1} \ll R_t \ll e^{\alpha t/2d}$.
\end{proposition}
One of the central steps of the proof is to extend an idea introduced in~\cite{ezanno12,andjelrolla}, which works using Harris' graphical construction. Conditioned on the existence of an infected point at time $t$, we first locate a site $x$ whose infection percolates from time $0$ to time $t$.
From site $x$ one cannot apply Proposition~\ref{prop:existencenu} directly, because choice of $x$ itself requires extra information that interferes with the process distribution. In order to clean most of this inconvenient information, a kind of renewal space-time point is used.

\smallskip

About the constant $\rho$, in~\cite{sturmswart} it is shown that the limiting distribution $\nu$ in~(\ref{eq:yaglom}) is the unique QSD satisfying
\(
\E_\nu |\zeta|<\infty.
\)
We will further prove the following.
\begin{proposition}
\label{prop:convexpectyaglom}
Under the extra assumption that $t^{2d+1} \ll R_t \ll e^{\sqrt[3]{t}}$,
\begin{align}
\nonumber
\E\left( | \eta_t\cap B_{R_t}| \,\big|\, \eta_t\cap B_{R_t}\neq \emptyset\right)\rightarrow \E_\nu | \zeta |
.
\end{align}
\end{proposition}

On the other hand, it is shown in~\cite{andjelrolla} that
the sub-Markovian kernel of the evolution $\zeta_t$ restricted to $\Lambda$ is $\alpha$-recurrent with summable left eigenmeasure $\nu$ and positive right eigenfunction $h$.
As a consequence, for each $A\in \Lambda$, one has not only~(\ref{eq:decexpo})
but also
\begin{equation}
\label{eq:hprefactor}
e^{\alpha t} \P \left( \zeta_t^A \ne \emptyset \right) \to h(A).
\end{equation}

Combining Propositions~\ref{prop:cvtonu} and~\ref{prop:convexpectyaglom} with the above limit, we will show that
\begin{equation}
\label{eq:rho}
\rho = \frac{h(\{\o\})}{\E_\nu |\zeta|}
.
\end{equation}

\bigskip

The rest of the paper is divided as follows.
In Section~\ref{sec:prelim} we briefly describe a graphical construction, introduce notation, state the FKG inequality, and give a definition of good points.
In Section~\ref{sec:infected} we prove Propositions~\ref{prop:cvtonu} and~\ref{prop:convexpectyaglom}.
In Section~\ref{sec:scaling} we prove Theorem~\ref{thm:convergence}.

\section{Preliminaries and main tools}
\label{sec:prelim}

We start gathering common notation for the reader's convenience. In \textsection\ref{sec:fkg} we will present Harris' graphical construction with further notation, and in \textsection\ref{sec:goodpoint} we will define good points, which will be a central concept used in the rest of the paper.

From~(\ref{eq:decexpo}),~(\ref{eq:hprefactor}) and duality, we have the useful estimates:
\begin{equation}
\label{eq:usefulestimates}
e^{-\alpha t} \geq \P(\eta_t^x \ne \emptyset) = \P(x\in\eta_t^{\Z^d}) \sim h(\{\o\}) e^{-\alpha t}
\qquad
\text{ as }
t \to \infty
.
\end{equation}

The
letter $\beta$ denotes an arbitrary number that will be enlarged throughout the proof.
The product $\beta t$ means $\lfloor \beta t \rfloor$. Some definitions such as that of good point are implicitly parametrized by $\beta$.

For $r>0$, we define the balls $\mathcal{B}_r=\{x\in\R^d: \|x\|_\infty\leq r\}$ and $B_r= \mathcal{B}_r\cap\Z^d$.
We also define the discrete sphere $D_r$ by $B_r \setminus B_{r-1}$. The associated balls and discrete sphere centered at $y$ are denoted by $\mathcal{B}_r^y$, $B_r^y$ and $D_r^y$.

We denote by $|A|$ the cardinal of $A \subset \Z^d$ and by $|\Psi|$ the Lebesgue measure of $\Psi \subset \R^d$.
The \emph{total variation distance} of probability measures will be denoted $\| \mu - \mu' \|=\sup_{A}{|\mu(A)-\mu'(A)|}$.

The letters $c$ and $C$ denote different positive and finite numbers each time they appear. They may depend on $\lambda$ and $d$, but neither on $t$ nor on $R_t$. When they depend on $\beta$ we write $c_\beta$ and $C_\beta$.
Finally, $a\sim b$ means $\frac{a}{b}\to 1$, $a \lesssim b$ means $\limsup \frac{a}{b} \leq 1$, $a \ll b$ means $\frac{a}{b}\to 0$, and $a \approx b$ means $\|a-b\| \to 0$.

\subsection{Graphical representation and FKG inequality}
\label{sec:fkg}

Define $\LL^d = \Z^d + \{\pm \frac{1}{3} e_i,i=1,\dots,d\}$, and let $U$ be a Poisson point process in $\R^{d}\times \R$ with intensity given by
\(
\big( \sum_{y\in\Z^d} \delta_y + \sum_{y\in\LL^d} \lambda\delta_y \big) \times \dd t
.
\)
Notice that $U \subseteq (\Z^d\cup \LL^d )\times \R$.
Let $(\Omega, \F, \P)$ be the underlying probability space.
For nearest neighbors $x,y \in \Z^d$ we write $U^{x,y} = \{t: (x+ \frac{y-x}{3},t) \in U\}$ and $U^x = \{t: (x,t) \in U\}$.
For $t\in U^x$ we say that there is a \emph{recovery mark} at site $x$ at time $t$, and for $t\in U^{x,y}$ we say that there is an \emph{infection arrow} from $x$ to $y$ at time $t$.

Given two space-time points $(y,s)$ and $(x,t)$, we define a \emph{path from $(y,s)$ to $(x,t)$} as a finite sequence $(x_0,t_0), \dots, (x_k,t_k)$
with $x_0=y$, $x_k=x$,
$s=t_0 \le t_1 \le \dots \le t_k = t$ with the following property.
For each $i=1,\dots,k$, the $i$-th segment $[(x_{i-1},t_{i-1}),(x_i,t_i)]$ is either \emph{vertical}, that is, $x_{i}=x_{i-1}$, or \emph{horizontal}, that is, $\|x_{i}-x_{i-1}\|_1=1$ and $t_{i}=t_{i-1}$.
Horizontal segments are also referred to as \emph{jumps}.
If all horizontal segments satisfy $t_i = t_{i-1} \in U^{x_{i-1},x_i}$ then such path is also called a \emph{$\lambda$-path}.
If, in addition, all vertical segments satisfy $(t_{i-1},t_{i}] \cap U^{x_i}=\emptyset$ we call it an \emph{open path} from $(y,s)$ to $(x,t)$.

The existence of an open path from $(y,s)$ to $(x,t)$ is denoted by
$ (y,s) \rightsquigarrow (x,t)$.
For two sets $C$, $D\subseteq \Z^d\times \R$, we use
$\{C \rightsquigarrow D\} = \{ (y,s) \rightsquigarrow (x,t) \text{ for some } (y,s)\in C, (x,t)\in D\}$.
We denote by $L_t$ the set $\Z^d \times \{t\}$.

For $A \subseteq \Z^d$, we define $\eta_{s,t}^A$ by
\begin{equation}
\label{def.contact.process}
\eta_{s,t}^A = \{ x\in\Z^d : (A\times\{s\}) \rightsquigarrow (x,t)\}.
\end{equation}

When $s=0$ we omit it in the subindex. We use $(\eta_t)_{t\ge0}$ for the process defined by~(\ref{def.contact.process}) with $A=\Z^d$, so $(\eta_t)$ is a contact process with parameter $\lambda$.
The process $(\zeta_t)_{t\ge0}$ given by $\zeta_t = \langle \eta_t \rangle$ is this contact process modulo translations. Both of them are Markov. Note that if $A$ is finite, the same holds for $\eta_t^A$ and $\zeta_t^A$ for every $t\ge 0$, almost surely. Also note that $\emptyset$ is absorbing for both processes. When $A$ is a singleton~$\{y\}$ we write $\eta^y_t$ and $\zeta^y_t$.

We use $\omega$ for a \emph{configuration} of points in $\R^{d+1}$ and $\omega_\delta$, $\omega_\lambda$ for its restrictions to $\Z^d\times \R$ and $\LL^d\times \R$, respectively.
We write $\omega\ge \omega'$ if $\omega_\lambda \supseteq \omega'_\lambda$ and $\omega_\delta \subseteq \omega'_\delta$.
We slightly abuse the notation and identify a set of configurations~$Q$ with the event ``$U \in Q$.''

A minor topological technicality needs to be mentioned.
Consider the space of locally finite configurations with the Skorohod topology: two configurations are close if they have the same number of points in a large space-time box and the positions of the points are approximately the same.
In the sequel we assume that all events considered have zero-probability boundaries under this topology.
The important fact is that events of the form $\{E \rightsquigarrow F \}$ are measurable and satisfy this condition, as long as $E$ and $F$ are closed subsets of $\Z^d \times \R$.

\begin{definition*}
A set of configurations $Q$ is \emph{increasing} if $\omega \ge \omega' \in Q$ implies $\omega \in Q$.
\end{definition*}

\begin{theorem*}
[FKG Inequality]
If $Q_1$ and $Q_2$ are increasing, then
\(
 \P(Q_1 \cap Q_2) \ge \P(Q_1)\P(Q_2).
\)
\end{theorem*}

See \cite[\textsection\textsection 2.1-2.2]{bezgrimexpo} for proofs and precise definitions.

\subsection{Good point}
\label{sec:goodpoint}

We say that the space-time point $(z,s)$ is a \emph{good point} if every $\lambda$-path starting from $(z,s)$ makes less than $\beta t $ jumps during $[s,s+t]$, and we denote by $G_z^s$ the corresponding event.
The definition of good point depends on $\beta$ and $t$ but we omit them in the notation $G_z^s$.
We define
\begin{align*}
G^I(A) = \text{``} G_x^s \text{ occurs for all } {x\in A}, s\in I \text{''}
,
\quad
\widetilde{G}_z^s = G^s(D_{ 2\beta t}^z)
\quad
\text{and}
\quad
\widehat{G}_z^s = \text{``}G_z^s \text{ and }\widetilde{G}_z^s \text{''}
.
\end{align*}

\begin{lemma}
For every $\rho<\infty$, for $\beta$ and $t$ large enough depending on $\rho$,
\[\P({G}_{\o}^0)\geq 1-e^{-\rho t}.\]
\end{lemma}

\begin{proof}
The proof in \cite[Lemma~4]{andjelrolla} works on any dimension.
\end{proof}

\begin{corollary}
\label{cor:goodwithout}
For every $\rho<\infty$, for $\beta$ and $t$ large enough depending on $\rho$,
\[\P(\widehat{G}_{\o}^0)\geq 1-e^{-\rho t}.\]
\end{corollary}
\begin{proof}
Use union bound over $D_{2\beta t}$ and increase $\beta$ to obtain the desired $\rho$.
\end{proof}

\begin{corollary}
\label{cor:goodconditoned}
Let $\rho<\infty$. If $\beta$ and $t$ are large enough, then
\[
\P\left(G^{[0,t]}(A) \,\middle|\, \big. \eta_s \cap B \neq \emptyset \right)\geq 1-e^{-\rho t}
\quad
\text{ and }
\quad
\P\left(G^{[0,t]}(A) \,\middle|\, \big. \eta_s^B \neq \emptyset \right)\geq 1-e^{-\rho t}
\]
for every pair of sets $A$ and $B$ with $|A|\leq e^{\rho t}$ and $|B|\ge 1$, and every $s\le t$.
\end{corollary}
\begin{proof} Let $\rho<\infty$ fixed. For $t$ and $\beta$ large enough, we have $\P\left(G^0_x \right)\geq 1-e^{-(2\rho+\alpha) t}$.
Using this and~(\ref{eq:usefulestimates}) we get
\begin{align*}
\P\left(G^0(A)\,\big|\,\eta_s \cap B \neq \emptyset\right)
&\geq 1-\frac{\P( \neg G^0(A))}{\P(\eta_s \cap B \neq \emptyset)}\geq 1-\frac{e^{\rho t} \cdot e^{-(2\rho+\alpha) t}}{e^{-\alpha t}}= 1-e^{-\rho t}
.
\end{align*}
Finally, we observe that, if $(x,0)$ is $(\beta,2t)$-good, then $(x,s)$ is $(2\beta,t)$-good for all $s\in[0,t]$. This proves the first inequality. The second one is proved the same way.
\end{proof}

The fact that good points have so high probability creates independence between sets of sites infected by $2\beta t$-distant points, and makes it possible to obtain the following bounds.

\begin{corollary}
\label{lowboundsurvival}
If $R_t \ll e^{\alpha t/d}$ then, for $t$ large enough, we have
\[\P\left(\eta_t\cap B_{R_t} \neq \emptyset\right)\geq c \, \frac{R_t^d}{t^{d}} \, e^{-\alpha t}.\]
\end{corollary}
\begin{proof}
Let $\rho>\alpha$ and choose $\beta$ as in Corollary~\ref{cor:goodconditoned}.
Take $A=\{2\beta t y:y\in B_{(R_t-3\beta t)/2\beta t}\}$.
Conditioning on $G^0(A)$, that is, the event where all points in $A$ are good, survivals of points in $A$ are independent (because two points in $A$ are at distance at least $2\beta t$).
Since $\P(G^0(A)) \to 1$ at $t\to\infty$, we get
\begin{align*}
\P\left(\eta_t\cap B_{R_t}\neq\emptyset\right)
& \gtrsim
\P\left(\eta_t^A\cap B_{R_t}\neq\emptyset\,\middle|\, G^0(A)\right)
\\
& =
1 - \prod_{x\in A} \P\big( \eta_t^x = \emptyset \,\big|\, G^0(A) \big)
\\
& \geq
1 - \prod_{x\in A} 1 -
\big[
\underbrace{\P\big( \eta_t^x \ne \emptyset \big)}_{\sim h(\{\o\})e^{-\alpha t}}
-
\underbrace{\P\big(\neg G^0(A)\big)}_{\lesssim e^{-\rho t}}
\big]
\\
& \sim
|A|h(\{\o\})e^{-\alpha t}
,
\end{align*}
because $e^{-\rho t} \ll e^{-\alpha t}$ and $|A|e^{-\alpha t} \ll 1$.
We conclude using $|A|\sim (R_t/\beta t)^d$.\end{proof}
\begin{corollary}
\label{cor:boundring}
If $R_t \ll e^{\alpha t/d}$ then, for $\beta$ and $t$ large enough, we have
\[\P\left(\eta_t\cap B_{R_t}\setminus B_{R_t-3\beta t}\neq\emptyset\,\big|\,\eta_t\cap B_{R_t}\neq\emptyset\right)\leq C_{\beta} \frac{t^{d+1}}{R_t}.\]
\end{corollary}
\begin{proof}
Using Corollary~\ref{lowboundsurvival} and~(\ref{eq:usefulestimates}),
\begin{multline*}
\P\left(\eta_t\cap B_{R_t}\setminus B_{R_t-3\beta t}\neq\emptyset\,\big|\,\eta_t\cap B_{R_t}\neq\emptyset\right)
\leq
\\
\leq
\frac{|B_{R_t}\setminus B_{R_t-3\beta t}|\cdot \P(x\in\eta_t)}{\P\left(\eta_t\cap B_{R_t}\neq\emptyset\right)}
\leq \frac{2d(2R_t)^{d-1}(3\beta t)e^{-\alpha t}}{c\, R_t^d t^{-d} e^{-\alpha t}}=C_{\beta}\frac{t^{d+1}}{R_t}.
\tag*{\qedhere}
\end{multline*}
\end{proof}

\section{Configuration at infected regions}
\label{sec:infected}

In this section we prove Propositions~\ref{prop:cvtonu} and~\ref{prop:convexpectyaglom}. In \textsection\ref{sec:minimalpath} we describe the construction of a minimal path, and in \textsection\ref{sec:breakpoints} we define a kind of renewal point on this path, called break point. We then state some important properties of the the break point and use them to prove Proposition~\ref{prop:cvtonu}. In \textsection\ref{sec:propbreakpoints} we prove the statistical properties about the break point, in \textsection\ref{sec:bpexistence} we prove properties about its space-time location and finally in \textsection\ref{sec:expectation} we prove Proposition~\ref{prop:convexpectyaglom}.

If the reader is wondering why the $R_t \ll e^{\alpha t/2d}$ requirement for Proposition~\ref{prop:cvtonu}, we note that the same argument works for $R_t \ll e^{\frac{\alpha t}{d+\epsilon}}$. Likewise, the assumption $R_t \ll e^{\sqrt[3]{t}}$ for Proposition~\ref{prop:convexpectyaglom} can be lessened to $R_t \ll e^{o(t/\log t)}$.
In any case, these assumptions are immaterial because, once Theorem~\ref{thm:convergence} is proved for some $1 \ll R_t \ll e^{\alpha t/d}$, the result can be bootstrapped to any other such sequence.

\subsection{Construction of the work path}
\label{sec:minimalpath}

Let $\precc$ denote a well-order of $\Z^d$, and read $w \prec z$ as $w$ has higher priority or precedes $z$.
In the following we take $A \subseteq \Z^d$ a non-empty initial configuration. Let $\D$ be a closed subset of $\Z^d\times\R^+$; we want to define a ``minimal'' path $\Gamma^{A\to \D}$ from $A\times\{0\}$ to $\D$.

If $A \times \{0\} \rightsquigarrow \D$, define
\[
X^{A\to \D}=\min\{x\in A, (x,0) \rightsquigarrow \D\} \in \Z^d
\]
as the highest-priority site whose infection survives up to the space-time region $\D$. We want to define a ``minimal path'' in a convenient way, which is slightly more delicate than in the one dimensional case~\cite{andjelrolla}.

The path $\Gamma^{A\to \D}$ will be described by a finite sequence $\{(x_0,t_0),\ldots,(x_n,t_n)\}$.
Take $x_0=X^{A\to \D}$ and $t_0=0$.
Notice that $(x_0,t_0^+) \rightsquigarrow \D$.
Suppose $(x_k,t_k)$ has been defined, and satisfies $(x_k,t_k^+) \rightsquigarrow \D$.
Take $\tau_{k+1}>t_k$ as the instant of the next infection arrow from $x_k$.
If $\{x_k\}\times [t_k,\tau_{k+1}]$ intersects $\D$, we finish the construction by taking $x_{k+1}=x_k$,
$t_{k+1}=\min\{ t\ge t_k : (x_k,t)\in \D \}$,
and $n=k+1$.

Otherwise, in order to continue the construction, we let $t_{k+1}=\tau_{k+1}$ and choose $x_{k+1}$ as follows.
Let $y_k$ be the destination of the outgoing arrow from $x_k$ at time $t_{k+1}$

Since $(x_k,t_{k+1})\rightsquigarrow \D$, either $(y_k,t_{k+1})\rightsquigarrow \D$ or $(x_k,t_{k+1}^+)\rightsquigarrow \D$ must occur.
If only one occurs, we choose the corresponding site.
If both occur, we choose the site with higher priority.
More precisely, we have four cases (illustrated in Figure~\ref{fig:minimalpath}):

\begin{enumerate}
\item[(a)] if $y_k \prec x_k$ and $(y_k,t_{k+1})\rightsquigarrow \D$, then $x_{k+1}=y_k$,
\item[(b)] if $y_k \prec x_k$ but $(y_k,t_{k+1})\not\rightsquigarrow \D$, then $x_{k+1}=x_k$,
\item[(c)] if $x_k \prec y_k$ and $(x_k,t_{k+1}^+)\rightsquigarrow \D$, then $x_{k+1}=x_k$,
\item[(d)] if $x_k \prec y_k$ but $(x_k,t_{k+1}^+)\not\rightsquigarrow \D$, then $x_{k+1}=y_k$.
\end{enumerate}
We finally define the path $\Gamma^{A\to\D}:[0,t_{n}]\to \Z^d$ by
$\Gamma^{A\to \D}(s)=x_k$,
$s\in [t_k,t_{k+1})$.

In the following, we restrict to $A=\Z^d$.
If $\eta_t\cap B_{R_t}\neq\emptyset$, we write $\Gamma=\Gamma^{\Z^d \to B_{R_t}\times\{t\}}$.
For $y\in\eta_s$, we let $\Gamma_{y,s}=\Gamma^{\Z^d \to (y,s)}$. We also write $X$ and $X_{y,s}$ for the respective starting points.

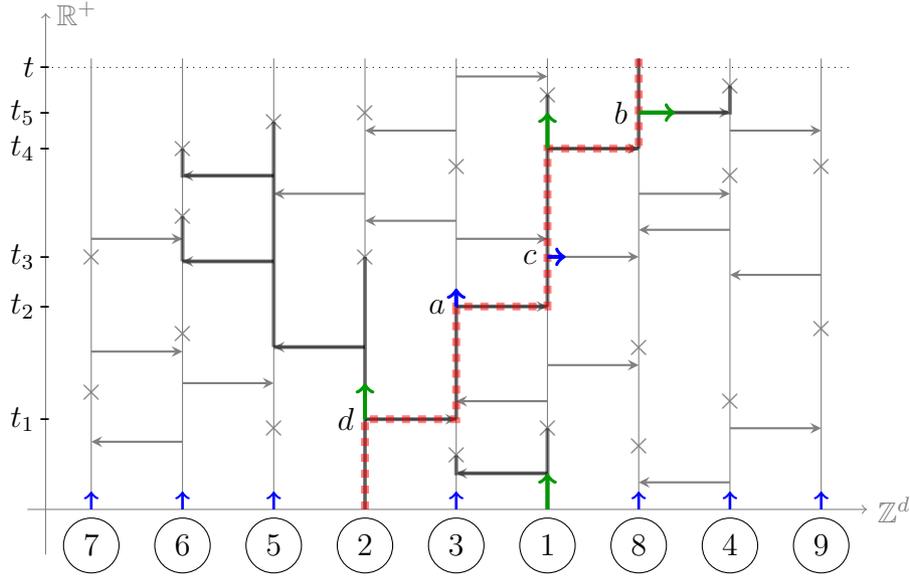
\begin{figure}[hb!]
{\centering\begin{tikzpicture}[scale=1.2]

\draw [->,gray] (-3.7,0)--(5.5,0);
\draw [->,gray] (-3.5,-0.5)--(-3.5,5.5);
\draw (-3.5,5.5) node[right,gray]{\small{$\R^+$}};
\draw (5.5,0) node[right,gray]{\small{$\Z^d$}};
\draw[dotted] (-3.5,4.9)--(5.3,4.9);
\draw (-3.5,4.9) node{-} node[left]{$t$};
\draw (-3.5,1) node{-} node[left]{$t_1$};
\draw (-3.5,2.25) node{-} node[left]{$t_2$};
\draw (-3.5,2.8) node{-} node[left]{$t_3$};
\draw (-3.5,4) node{-} node[left]{$t_4$};
\draw (-3.5,4.4) node{-} node[left]{$t_5$};
\foreach \x in {-3,-2,-1,0,1,2,3,4,5}
\draw[gray] (\x,0)--(\x,5);

 \foreach \x in {(-3,1.3),(-3,2.8),(-2,1.95),(-2,3.25),(-1,0.9),(-1,4.3),(0,2.8),(1,3.8),(1,0.6), (2,4.6),(2,0.9),(3,0.7),(3,1.8), (4,1.2), (4,3.7),(0,4.4),(-2,4),(4,4.7),(5,2),(5,3.8)}
\draw[gray] \x  node{$\times$};

\foreach \x/\y in { {(-3,1.75)}/{(-2,1.75)}, {(-3,3)}/{(-2,3)}, {(-2,1.4)}/{(-1,1.4)},{(0,1)}/{(1,1)},{(1,2.25)}/{(2,2.25)},{(1,3)}/{(2,3)},{(1,4.8)}/{(2,4.8)},{(2,1.6)}/{(3,1.6)},{(2,2.8)}/{(3,2.8)},{(2,4)}/{(3,4)},{(3,3.5)}/{(4,3.5)},{(3,4.4)}/{(4,4.4)},{(4,0.9)}/{(5,0.9)},{(4,4.2)}/{(5,4.2)}}
\draw [->,>=stealth,thick,gray] \x--\y;

\foreach \x/\y in { {(-3,0.75)}/{(-2,0.75)},  {(-2,2.75)}/{(-1,2.75)},{(-2,3.7)}/{(-1,3.7)},{(-1,1.8)}/{(0,1.8)},{(-1,3.5)}/{(0,3.5)},{(0,3.2)}/{(1,3.2)},{(0,4.2)}/{(1,4.2)},{(1,0.4)}/{(2,0.4)},{(1,1.2)}/{(2,1.2)},{(3,0.3)}/{(4,0.3)},{(3,3.1)}/{(4,3.1)},{(4,2.6)}/{(5,2.6)}}
\draw [<-,>=stealth,thick,gray] \x--\y;


\node[draw,circle] (6) at (-2,-0.4){6};
\node[draw,circle] (5) at (-1,-0.4){5};
\node[draw,circle] (7) at (-3,-0.4){7};
\node[draw,circle] (2) at (0,-0.4){2};
\node[draw,circle] (1) at (2,-0.4){1};
\node[draw,circle] (3) at (1,-0.4){3};
\node[draw,circle] (8) at (3,-0.4){8};
\node[draw,circle] (4) at (4,-0.4){4};
\node[draw,circle] (9) at (5,-0.4){9};
\node[left] at (0,1){$d$};
\node[left] at (1,2.25){$a$};
\node[left] at (3,4.4){$b$};
\node[left] at (2,2.8){$c$};

\foreach \a/\b in {{(0,0)}/{(0,1)},{(0,1)}/{(1,1)},{(0,1)}/{(0,1.8)},{(1,1)}/{(1,2.25)},{(0,1.8)}/{(-1,1.8)},{(1,2.25)}/{(2,2.25)},
{(-1,1.8)}/{(-1,2.75)},{(0,1.8)}/{(0,2.8)},{(-1,2.75)}/{(-2,2.75)},
{(2,2.25)}/{(2,2.8)},
{(2,2.8)}/{(2,4.6)},{(3,4)}/{(3,5)},{(3,4)}/{(2,4)},
{(-1,2.75)}/{(-1,3.5)},{(-2,2.75)}/{(-2,3.25)},
{(-1,3.5)}/{(-1,4.3)}}
\draw[black,line width=1.5pt,opacity=0.5] \a--\b;
\draw[black,line width=1.5pt,opacity=0.5] (-1,3.7)--(-2,3.7)--(-2,4);
\draw[black,line width=1.5pt,opacity=0.5] (3,4.4)--(4,4.4)--(4,4.7);

\draw[black,line width=1.5pt,opacity=0.5] (2,0)--(2,0.9);
\draw[black,line width=1.5pt,opacity=0.5] (2,0.4)--(1,0.4)--(1,0.6);
\draw[mygreen,->,line width=1.5pt,opacity=1] (2,0)--(2,0.4);

\draw[red,line width=3pt,dashed,opacity=0.5] (0,0)--(0,1)--(1,1)--(1,2.25)--(2,2.25)--(2,4)--(3,4)--(3,5);

\draw[mygreen,->,line width=1.5pt,opacity=1] (0,1)--(0,1.4);
\draw[mygreen,->,line width=1.5pt,opacity=1] (3,4.4)--(3.4,4.4);
\draw[mygreen,->,line width=1.5pt,opacity=1] (2,4)--(2,4.4);
\draw[color=blue,->,line width=1.5pt,opacity=1] (2,2.8)--(2.2,2.8);
\draw[color=blue,->,line width=1.5pt,opacity=1] (1,2.25)--(1,2.45);

\draw[color=blue,->,line width=1pt,opacity=1] (-3,0)--(-3,0.2);
\draw[color=blue,->,line width=1pt,opacity=1] (-2,0)--(-2,0.2);
\draw[color=blue,->,line width=1pt,opacity=1] (-1,0)--(-1,0.2);
\draw[color=blue,->,line width=1pt,opacity=1] (1,0)--(1,0.2);
\draw[color=blue,->,line width=1pt,opacity=1] (3,0)--(3,0.2);
\draw[color=blue,->,line width=1pt,opacity=1] (4,0)--(4,0.2);
\draw[color=blue,->,line width=1pt,opacity=1] (5,0)--(5,0.2);
\end{tikzpicture}\par}
\caption{(Color online)
Minimal path.
The sites are numbered according to the order $\precc$. \ding{193} is the highest priority point connected to $L_t$. Among all paths from $L_0$, the curve $\Gamma$ (in dashed red) is the ``minimal'' path according to the rules (a), (b), (c), and (d). At (\ding{193},$t_1$), we have an arrow to \ding{194} which leads to a lower priority point but we still follow the arrow because (\ding{193},$t_1^+$) is not connected to $L_t$ (case d). At (\ding{194},$t_2$), we have an arrow to \ding{192} which leads to a higher priority point and this one is connected to $L_t$ (case a). At (\ding{192},$t_3$), there is an arrow to \ding{199} which leads to a lower priority point and (\ding{192},$t_3^+$) is connected to $L_t$ (case c) so we do not follow the arrow. At (\ding{199},$t_5$), there is an arrow to \ding{195} which has higher priority but (\ding{195},$t_5$) is not connected to $L_t$ (case b).}
\label{fig:minimalpath}
\end{figure}

\subsection{Break point and consequences}
\label{sec:breakpoints}

The space-time point $(y,s)$ is called \emph{a break point} if for every $x\in B_{2\beta t}^y\setminus\{y\}$, $L_0 \not\rightsquigarrow (x,s)$. Suppose $\eta_t\cap B_{R_t}\neq \emptyset$, so that $\Gamma$ is well defined. Let
\[S=\inf\{\, s \in [0,t]: s=t \text{ or }(\Gamma(s),s)\text{ is a break point} \,\}\]
be the time of the first break point on $\Gamma$, and let $Y=\Gamma(S)$ be its spatial location.
We will refer to $(Y,S)$ as \emph{the break point}.

This definition provides control on the configuration $\eta_t$ through the following lemmas.

\begin{lemma}
\label{lemma:stepone}
For all $s\in [0,t)$ and $y\in B_{R_t-\beta t}$,
\[\law \left(\eta_t\cap B_{\beta t}^y \,\big|\, Y=y, S=s, \widehat{G}_y^s\right) = \law \left(\eta_{t-s}^{y} \,\big|\, y\leadsto L_{t-s}, G_y^0\right).\]
\end{lemma}

\begin{lemma}
\label{lemma:steptwo}
For all $s\in [0,t]$ and $y\in\Z^d$,
\[\E\left( |\eta_t \cap B_{R_t} \setminus B_{\beta t}^y| \,\big|\, Y=y, S=s, \widehat{G}_y^s\right) \leq |B_{R_t}| e^{-\alpha(t-s)}.\]
\end{lemma}

We will need some control on the properties of $Y$ and $S$.

\begin{lemma}
\label{lemma:breakpointgood}
If $R_t = e^{O(t)}$ then, for $\beta$ and $t$ large enough,
\[
\P\left( \widehat{G}^Y_S \,\middle|\, \big. \eta_t\cap B_{R_t}\neq \emptyset\right)\geq 1-e^{- c t}.
\]
\end{lemma}

\begin{lemma}
\label{lemma:probabreak1}
If $R_t = e^{O(t)}$ then, for $\beta$ and $t$ large enough,
\[ \P\left(S\leqslant \tfrac{t}{2}\,\middle|\,\eta_t\cap B_{R_t}\neq\emptyset\right)\ge 1-e^{-c_{\beta}\sqrt{t}}.\]
\end{lemma}

\begin{lemma}
\label{lemma:probabreak2}
If $R_t \ll e^{\alpha t/d}$ then, for $\beta$ and $t$ large enough,
\[\P\left(Y\in B_{R_t-\beta t}\,\middle|\, \big. \eta_t\cap B_{R_t}\neq \emptyset\right)\geq 1-C_{\beta}\frac{t^{d+1}}{R_t}.\]
\end{lemma}

We now have all the ingredients to prove Proposition~\ref{prop:cvtonu}.

\begin{proof}
[Proof of Proposition~\ref{prop:cvtonu}]
Let $R_t$ satisfy $t^{d+1} \ll R_t \ll e^{\alpha t/2d}$.
\begin{align}
\nonumber
\hspace{2mm} & \hspace{-2mm}
\limsup_{t\to\infty} \left\| \law\left(\langle\eta_t\cap B_{R_t}\rangle\,\big|\, \eta_t\cap B_{R_t}\neq \emptyset\right)-\nu \right\|
\\&
\nonumber
\phantom{a}\leq \limsup_{t\to\infty} \left\| \law\left(\langle\eta_t\cap B_{R_t}\rangle\,\big|\,\widehat{G}_Y^S, S\leq t/2,Y\in B_{R_t-\beta t}\right)-\nu\right\|
+
\\&
\qquad
\qquad
+
\limsup_{t\to\infty} \left[ 1-\P\left( \widehat{G}_Y^S, S\leq t/2,Y\in B_{R_t-\beta t} \,\big|\, \eta_t\cap B_{R_t}\neq \emptyset \right)\right]
\label{eqA}
\\&
\nonumber
\phantom{a}
\leq \limsup_{t\to\infty} \sup_{\ontop{y\in B_{R_t-\beta t}}{s\in[0,t/2]}} \left\|\law\left(\langle\eta_t\cap B_{R_t}\rangle\,\big|\,\widehat{G}_y^s, Y=y, S=s\right)-\nu\right\|
\\&
\phantom{a}
\leq \limsup_{t\to\infty} \sup_{\ontop{y\in B_{R_t-\beta t}}{s\in[0,t/2]}}
\left[
\left\|\law\left(\langle\eta_t\cap B_{\beta t}^y\rangle\,\big|\,\widehat{G}_y^s, Y=y, S=s\right)-\nu\right\|
\label{eqB}
+|B_{R_t}|e^{-\alpha(t-s)}
\right]
\\&
\phantom{a}
= \limsup_{t\to\infty} \sup_{s\in[0,t/2]}\left\|\law\left(\langle\eta_{t-s}^0\rangle\,\big|\,0\leadsto L_{t-s}, G_0^0\right)-\nu\right\|
\label{eqC}
\\&
\label{eq:lastone}
\phantom{a}
\leq
\limsup_{t\to\infty} \sup_{s\in[0,t/2]}\left\|\law\left(\langle\eta_{t-s}^0\rangle\,\big|\,0\leadsto L_{t-s}\right)-\nu\right\|
=0
.
\end{align}
The term in~(\ref{eqA}) equals zero by Lemmas~\ref{lemma:breakpointgood},~\ref{lemma:probabreak1} and~\ref{lemma:probabreak2}.
To obtain~(\ref{eqB}) we use Lemma~\ref{lemma:steptwo}.
The second term vanishes by the assumption on $R_t$.
Equality~(\ref{eqC}) follows from Lemma~\ref{lemma:stepone}.
In~(\ref{eq:lastone}), we use Corollary~\ref{cor:goodconditoned} and Proposition~\ref{prop:existencenu}.
\end{proof}

\subsection{Properties of the break point}
\label{sec:propbreakpoints}

We now prove Lemmas~\ref{lemma:stepone} and~\ref{lemma:steptwo}.

\begin{proof}
[Proof of Lemma~\ref{lemma:stepone}]
Recall the definitions of $X$, $\Gamma$,
$Y$ and $S$ from \textsection\ref{sec:minimalpath}.
Let
\begin{align*}
E^i_{y,s}&= \mathcal{B}_{\beta t}^y \times (s,2t]
\\
E^e_{y,s}&=(\R^d\times [0,2t]) \setminus E^i_{y,s}
.
\end{align*}

In order to decompose the event $\{Y=y,S=s\}$ according to $E^i_{y,s}$ and $E^e_{y,s}$, we introduce the following events. Let
\begin{align*}
H_1&=\big\{ L_0 \rightsquigarrow (y,s)\big\}\\
H_2&=\{L_0 \not\rightsquigarrow (B_{2\beta t}^y\setminus\{y\})\times \{s\}\}\\
H_3&=\{ \forall u\in[0,s),L_0 \leadsto (B_{2\beta t}^{\Gamma_{y,s}(u)}\setminus \Gamma_{y,s}(u))\times\{u\} \, \}
.
\end{align*}
$H_1$ means that $(y,s)$ is reached, $H_2$ means that $(y,s)$ is a break point, $H_3$ means that there are no other break points in $\Gamma_{y,s}$.

\begin{figure}[hb!]
{\centering\begin{tikzpicture}

\draw[dashed] (-4,7)--(10,7);
\draw[->] (-4.2,0)--(-4.2,7.5) node[right]{$\mathbb R^+$};
\draw (-4,0) node[left]{$0-$};\draw (-4,3) node[left]{$s-$};
\draw (-4,7) node[left]{$t-$};

\draw (6.5,7) node[below]{$E^i_{y,s}$};
\draw (-4,0)--(10,0);
\draw (0,0) node{$|$} node[below]{$0$}; 
\draw (3,0)  node[below]{$X$};
\draw (6,3) node[below right]{$y$};
\fill[gray,opacity=0.3] (4.7,3)--(7.3,3)--(7.3,7)--(4.7,7)--cycle;
\draw [thick] (3.4,3)--(8.6,3);
\draw[<->] (3.4,3.1)--(6,3.1);
\draw (4.3,3) node[above]{$2\beta t$};
\draw[<->] (4.7,7.1)--(7.3,7.1) node[midway, above]{$y+B_{\beta t}$};
\foreach \x in {-3.5, 2,5.5, 1.5, 3}
\draw (\x,0) node{$\bullet$};
\draw [red] plot [only marks, mark=*] coordinates {(-3,3) (0,3) (1,3) };
\draw [black] plot [only marks, mark=square*] coordinates {(4,0) (-2,0) (-0.5,0) (8.5,0)};
\draw [red] plot [only marks, mark=square*] coordinates {(-1,3) (2,3)};
\draw [color=mygreen] plot [only marks, mark=square*] coordinates {(2.7,3)};
\draw [blue] plot [only marks, mark=*] coordinates {(9,3)};

\draw[->,>=latex,red] (-3.5,0) to[out=90,in=-120] (-3,3)to[out=60,in=-90] (-1,7);
\draw[red] (-3.5,0) to[out=90,in=-120] (-2,2);
\draw[red] (-3,3)to[out=60,in=-90] (-3.7,4.5);
\draw[->,>=latex,red] (2,0) to[out=90,in=-90] (0,3) to[out=90,in=-90] (1.2,7);
\draw[red] (-0.5,0) to[out=120,in=-120] (0,1.5);
\draw[red] (1.5,0) to[out=90,in=-90] (1.5,1.02);
\draw[red] (1.5,1.12) to[out=90,in=-120] (1,3)to[out=60,in=-120] (2,5);

\draw[red] (4,0) to[out=90,in=-30] (3.5,0.9);
\draw[red] (3.4,0.95) to[out=150,in=-90] (2,3);
\draw [->,>=latex,red,dashed] (2,3) to[out=90,in=-70,dashed] (3,7);
\draw[red] (-2,0) to[out=90,in=-90] (-1,3);
\draw[->,>=latex, dashed,red] (-1,3) to[out=90,in=-90] (-3,7) ;

\draw[color=mygreen] (4.2,1.35) to[out=150,in=-90] (2.7,3);
\draw [->,>=latex,color=mygreen,dashed] (2.7,3) to[out=90,in=-90,dashed] (3.8,7);
\draw[blue] (4.65,1.55) to[out=0,in=-120] (9,3);
\draw [->,>=latex,blue] (9,3) to[out=60,in=-90,dashed] (9.5,7);

\draw[->,>=latex, color=mypink, line width=1pt] (3,0) to[out=90,in=-90] (6,3);
\draw[->,>=latex, color=mypink, line width=1pt] (6,3) to[out=90,in=-70] (6.3,5.5);
\draw[->,>=latex, color=mypink, line width=1pt] (6.3,5.5) to[out=110,in=-110] (5.7,7);
\draw[->,>=latex,dashed] (5.5,0) to[out=90,in=-90] (7,3);
\draw[->,>=latex,dashed] (8.5,0) to[out=90,in=-90] (7.7,3);
\draw[red] (4.65,1.55) to[out=20,in=-90] (5,2.8);

\draw[->,>=latex, dotted] (3.4,3) to[out=90,in=-90] (4.6,7);
\draw[->,>=latex, dotted] (8.6,3) to[out=90,in=-90] (7.4,7);
\draw[->,>=latex, dotted] (6,3) to[out=90,in=-90] (4.8,7);
\draw[->,>=latex, dotted] (6,3) to[out=90,in=-90] (7.2,7);

\end{tikzpicture}\par}
\caption{(Color online)
The event $\{Y=y, S=s, \widehat{G}_y^s\}$. This event splits in three parts. The first one depends on what happens before time $s$. It consists of the occurrence of $H_1$, $H_2$, $H_3$. The set $\mathcal{A}''$ is represented by green and red squared points whereas the set $\mathcal{A}'\setminus \mathcal{A}''$ is represented by blue and red bullets. $H_1$ occurs by the first part of fuchsia path, $H_2$ consists of the non existence of the dashed black paths. Moreover $s$ is the first time with this property: $H_3$ consists of the existence of open paths from $L_0$ to $B_{2\beta t}^{\Gamma_{y,s}(u)}\setminus\Gamma_{y,s}(u)$ for all $u<s$. The rest depends on what happens between time $s$ and time $t$. The second part depends on $E^e_{y,s}$. It consists of the occurrence of $H_4$ that is the non existence of the green or red dashed paths from $\mathcal{A}''$ (squared points at time $s$) to $B_{R_t}$, and of the occurrence of ${\widetilde{G}_y^s}$: the dotted $\lambda$-paths starting from $D_{2\beta t}^y\times\{s\}$ cannot reach distance $\beta t$ by time $s+t$. The third part depends on the gray region $E^i_{y,s}$ and consists of both occurrences of $\{(y,s)\rightsquigarrow L_t\}$ (second part of fuchsia path) and $G_y^s$ that is the dotted paths starting in $(y,s)$ cannot reach $\beta t$ by time $t$.}
\label{figumbrella}
\end{figure}

We now need to describe necessary and sufficient conditions so that $(y,s)$ is on $\Gamma$, that is, so that $\Gamma_{y,s}$ and $\Gamma$ coincide on $[0,s]$.
To that end, we find a set of space-time points that must not percolate in order to ensure that $(y,s)$ is on $\Gamma$ (the red and green squares at time $s$ in Figure~\ref{figumbrella}).
Denote by $\{(x_1,t_1),\ldots,(x_n,t_n)\}$ and $\{y_1,\ldots,y_n\}$ the points of the construction procedure of $\Gamma_{y,s}$, as described in \textsection\ref{sec:minimalpath}. Let
\begin{align*}
\mathcal{A}_{y,s}^1&=\{(x_k,t_{k+1})\text{ such that } x_k \prec y_k \text{ but }(x_k,t_{k+1}^+)\not\rightsquigarrow (y,s)\}\\
\mathcal{A}_{y,s}^2&=\{(y_k,t_{k+1})\text{ such that } y_{k} \prec x_k\text{ but }(y_{k},t_{k+1})\not\rightsquigarrow (y,s)\}\\
\mathcal{A}_{y,s}^3 & = \{(x,0): x\prec X_{y,s} \}
\\
\mathcal{A}'' & = \{(z,s) : \text{ there exists }(x,u)\in \mathcal{A}_{y,s}^1\cup \mathcal{A}_{y,s}^2\cup \mathcal{A}_{y,s}^3 \text{ such that } (x,u^+)\rightsquigarrow (z,s)\}.
\end{align*}
The sets $\mathcal{A}_{y,s}^1$ and $\mathcal{A}_{y,s}^2$ are determined by the green arrows in Figure~\ref{fig:minimalpath}, corresponding to cases $d$ and $b$ respectively, while the set $\mathcal{A}_{y,s}^3$ consists of points that had higher priority than $X$ at time $0$.
They correspond to space-time points met during the construction of $\Gamma_{y,s}$ that had higher priority but were not chosen because they did not connect to $(y,s)$.
Finally, the set $\mathcal{A}''$ corresponds to the offspring of these points at time $s$.

Let $H_4$ be the event that $\mathcal{A}''\not\rightsquigarrow B_{R_t} \times \{t\}$.
We then have the following equivalence:
\begin{equation}
\label{eq:equivalence}
\{ Y=y, S=s, \widehat{G}_y^s\}= \{(y,s)\leadsto L_{t}, H_1,H_2,H_3,H_4, \widetilde{G}_y^s,G_y^s\}.
\end{equation}

Let $\mathcal{A}'=\eta_s \setminus\{y\}$; we have
\begin{align*}
\law& \left(\eta_t \cap B_{\beta t}^y \,\big|\, Y=y, S=s, \widehat{G}_y^s\right)
\\
&\phantom{aaaaaa}= \law \left(\eta_{s,t}^{\mathcal{A'}\cup \{y\}}\cap B_{\beta t}^y \,\big|\, (y,s)\leadsto L_{t}, H_1,H_2,H_3,H_4, \widetilde{G}_y^s,G_y^s\right)
\\
&\phantom{aaaaaa}= \law \left(\eta_{s,t}^{y}\cap B_{\beta t}^y\,\big|\, (y,s)\leadsto L_{t}, H_4, \widetilde{G}_y^s,{G}_y^s\right)
\end{align*}
because, on the one hand, $H_1,H_2,H_3$ depend on the graphical construction up to time $s$ and, on the other hand, occurrence of $\widetilde{G}_y^s$ and $G_y^s$ implies $\eta_{s,t}^{\mathcal{A}'\cup \{y\}}\cap B_{\beta t}^y = \eta^y_{s,t}$.

Now, $H_4 \cap \widetilde{G}_y^s$ depends on $E^e_{y,s}$, whereas $\{(y,s)\leadsto L_{t}\}\cap G_y^s$ depends on $E^i_{y,s}$.
Under the occurrence of $G_y^s$, the configuration $\eta_{s,t}^{y}\cap B_{\beta t}^y$ also depends only on $E^i_{y,s}$, so we have
\begin{multline*}
\law \left(\eta_t \cap B_{\beta t}^y \,\big|\, Y=y, S=s, \widehat{G}_y^s\right)
= \\
= \law \left( \eta_{s,t}^{y}\cap B_{\beta t}^y \,\big|\, (y,s)\leadsto L_{t}, G_y^s\right)
= \law \left(\eta_{t-s}^{y} \,\big|\, y\leadsto L_{t-s}, G_y^0\right)
.
\tag*{\qedhere}
\end{multline*}
\end{proof}

\begin{proof}
[Proof of Lemma~\ref{lemma:steptwo}]
Recall the construction and notations from the proof of Lemma~\ref{lemma:stepone}.
On the occurrence of $\left\{Y=y, S=s, \widehat{G}_y^s\right\}$, we know that $\eta_s =\mathcal{A'}\cup\{y\}$ with $\mathcal{A}'\cap B_{2\beta t}^y=\emptyset$. Given that $(y,s)$ and $D_{2\beta t}^y\times\{s\}$ are good points, we have:
\begin{align*}
\eta_t = \eta_{s,t}^{\mathcal{A}'\cup \{y\}}= \left(\eta_{s,t}^{y} \cap B_{\beta t}^y\right) \uplus \left(\eta_{s,t}^{\mathcal{A}'}\setminus B_{\beta t}^y\right).
\end{align*}

Using~(\ref{eq:equivalence}), we get
\begin{align*}
\E\left(\eta_t \cap B_{R_t} \setminus B_{\beta t }^y \,\big|\, Y=y, S=s, \widehat{G}_y^s\right)
& = \E\left( \eta_{s,t}^{\mathcal{A}'}\cap B_{R_t}\,\big|\,Y=y, S=s, \widehat{G}_y^s\right)\\
& = \E\left( \eta_{s,t}^{\mathcal{A}'}\cap B_{R_t}\,\big|\,(y,s)\leadsto L_{t}, H_4, \widetilde{G}_y^s,{G}_y^s\right)\\
& = \E\left( \eta_{s,t}^{\mathcal{A}'}\cap B_{R_t}\,\big|\,H_4,\widetilde{G}_y^s\right).
\end{align*}
The second equality follows from the fact that $H_1$, $H_2$ and $H_3$ depend on the graphical construction up to time $s$.
The last equality is due to the fact that, on the event $\widetilde{G}_y^s$, occurrence of $\eta_{s,t}^{\mathcal{A}'}\cap B_{R_t}\neq \emptyset$ depends on $E^e_{y,s}$, whereas the occurrence of both $(y,s)\leadsto L_{t}$ and $G_y^s$ depends on $E^i_{y,s}$. Now, since the random variable $\eta_{s,t}^{\mathcal{A}'}\cap B_{R_t}$ is increasing, and the occurrence of $H_4\cap \widetilde{G}_y^s$ is decreasing, by the FKG inequality we have
\[
\E\left(\eta_{s,t}^{\mathcal{A}'}\cap B_{R_t}\,\big|\,H_4,\widetilde{G}_y^s\right)
\leq \E\left( \eta_{s,t} \cap B_{R_t}\right)
\leq |B_{R_t}| e^{-\alpha(t-s)}.
\]
The last inequality is obtained applying~(\ref{eq:usefulestimates}).
\end{proof}

\subsection{Existence and location of the break point}
\label{sec:bpexistence}

We start by proving Lemmas~\ref{lemma:breakpointgood} and~\ref{lemma:probabreak2}.
Lemma~\ref{lemma:probabreak1} requires further definitions.

\begin{proof}
[Proof of Lemma~\ref{lemma:breakpointgood}]
Suppose $G^{[0,t]}(B_{R_t + 3 \beta t})$ occurs.
Since all sites in $B_{R_t + \beta t}$ are good, no open path can exit this box and then reach $B_{R_t}$ by time $t$, so $Y$ must be in this box and thus $B^Y_{2\beta t} \subseteq B_{R_t + 3 \beta t}$.
Since $S \in [0,t]$, this implies the occurrence of $\widehat{G}^Y_S$.
The lemma thus follows from Corollary~\ref{cor:goodconditoned}.
\end{proof}

\begin{proof}
[Proof of Lemma~\ref{lemma:probabreak2}]
On the occurrence of $G^0(D_{R_t-2\beta t})$,
the sites in $B_{R_t-3\beta t}$ cannot be infected by sites outside $B_{R_t-2\beta t}$ by time $t$.
Now if $\eta_t\cap B_{R_t}\setminus B_{R_t-3\beta t}=\emptyset$ also occurs,
this implies as well that $X \in B_{R_t-2\beta t}$ and that $Y \in B_{R_t-\beta t}$.

Now using Corollaries~\ref{cor:goodconditoned} and~\ref{cor:boundring} we have for $\beta$ and $t$ large enough
\begin{multline*}
\P \left(Y \not\in B_{R_t-\beta t}\,\big|\,\eta_t\cap B_{R_t}\neq\emptyset\right)
\le
\\
\le
\P \left( \neg G^0(D_{R_t-2\beta t}) \,\big|\, \eta_t\cap B_{R_t}\neq\emptyset \right)
+
\P \left( \eta_t\cap B_{R_t}\setminus B_{R_t-3\beta t}\neq\emptyset ,\big|\,\eta_t\cap B_{R_t}\neq\emptyset\right)
\le
\\
\le
e^{-t} + C_{\beta}\frac{t^{d+1}}{R_t}
,
\end{multline*}
and by adjusting $C_\beta$ we conclude the proof.
\end{proof}

\begin{definition*}
[Favorable time intervals]
Let $\gamma$ be a path in the time interval $[0,t]$ and let $\beta>0$.
We say that a time interval $[s-\sqrt{t},s)\subseteq [0,t]$ is \emph{favorable} for path $\gamma$ if for any $u\in [s-\sqrt{t} ,s)$ the number of jumps of $\gamma$ during $[u,s)$ is at most $4 \beta|s-u|$.
\end{definition*}

\begin{lemma}
\label{lemma:favorableintervals}
Let $\gamma$ be a path in the time interval $[0,t]$ with at most $\beta t$ jumps.
Then there are at least $\frac{\sqrt{t}}{4}-1$ disjoint favorable intervals for $\gamma$ contained in $[0,\frac{t}{2}]$.
\end{lemma}

For a proof, see \cite[Lemma~2.12]{andjelrolla}.

\begin{lemma}
\label{L7}
If a path $\gamma$ in the time interval $[0,t]$ has at least $k$ disjoint favorable intervals in $[0,\frac{t}{2}]$, then for all $t\ge 1$
\[ \mathbb P \left( S \leqslant \tfrac{t}{2} \,\middle|\, \Gamma=\gamma \right) \geq 1-p_\beta^k,\] with $p_\beta<1$ for all $\beta$.
\end{lemma}

\begin{proof}
The main difference with the proof of \cite[Lemma~2.15]{andjelrolla} is in the following step.
Let $D_\gamma$ be the closed set given by the union of the horizontal and vertical segments of $\gamma$ and $D_\gamma^c=(\R^d\times[0,t])\setminus D_\gamma$.
We will write the event $\{\Gamma=\gamma\}$ as
\[
\left\{\Gamma=\gamma\right\}=\{\gamma\text{ is open}\}\cap H_\gamma^c
,
\]
where the event $H_\gamma$ is increasing and depends only on $D_\gamma^c$. These properties depend on the definition of $\Gamma$ in a more subtle way than for $d=1$.

Recall the construction described in \textsection\ref{sec:minimalpath}. Suppose that a certain path $\gamma$ from $(x_0,0)$ to $B_{R_t}\times \{t\}$ is open. If $\gamma$ is the unique open path from $L_0$ to $B_{R_t}\times \{t\}$, then of course $\Gamma$ equals $\gamma$. On the other hand, when $\gamma$ is not the unique open path, for $\Gamma$ to be different from $\gamma$, it is necessary and sufficient that some path of ``higher priority'' exists. This can be a path starting from some $x \prec x_0$ at time $0$, a path starting as a horizontal jump away from $\gamma$ towards a site with higher priority, or a path starting as a continuation of $\gamma$ when the latter jumps towards a site of lower priority, as long as this alternative path either ends in $B_{R_t}\times \{t\}$ or meets with $\gamma$ at a later time. Defining $H_\gamma$ as the occurrence of at least one such open path, we see that, when $\gamma$ is open, $\Gamma=\gamma$ if and only if $H_\gamma$ does not occur.

The remainder of the proof is a straightforward adaptation of \cite[Lemma~2.15]{andjelrolla} to general dimensions, and is included in the sequel for convenience.

Let $E_t=(B_{2\beta t} \setminus\{\mathbf{0}\})\times\{0\}$. When $\Gamma=\gamma$, for any $(y,s)$ in $\gamma$,
\[
J_{\gamma,s,y}^c
\text{ implies that }
(y,s) \text{ is a break point}
,
\]
with
\begin{align*}
J_{\gamma,s,y} = \text{``}L_0\rightsquigarrow (y,s)+E_t\text{ in }D_\gamma^c \text{ or }\gamma\rightsquigarrow (y,s)+E_t\text{ in }D_\gamma^c\text{''}.
\end{align*}
Thus,
\begin{align*}
\mathbb P \left( S \leqslant \tfrac{t}{2} \,\big|\, \Gamma=\gamma \right)
&\ge
\mathbb P\left( J_{\gamma,s,\gamma(s)}^c \text{ for some }s\in[0,\tfrac{t}{2}] \,\big|\, \Gamma=\gamma \right)
\\
&=
\mathbb P\left( J_{\gamma,s,\gamma(s)}^c \text{ for some }s\in[0,\tfrac{t}{2}] \,\big|\, \gamma\text{ open },H_\gamma^c \right)
\\
&=
\mathbb P\left( J_{\gamma,s,\gamma(s)}^c \text{ for some }s\in[0,\tfrac{t}{2}] \,\big|\, H_\gamma^c \right).
\end{align*}
For the last equality, we used the fact that $\{\gamma\text{ is open}\}$ depends on $U\cap D_\gamma$ whereas $H_\gamma^c$ and $J_{\gamma,s,\gamma(s)}$ depend on $U\cap D_\gamma^c$.

Now, since the occurrence of $J_{\gamma,s,\gamma(s)}^c \text{ for some }s\in[0,\tfrac{t}{2}]$ as well as the occurrence of $H_\gamma^c$ are both decreasing, by FKG inequality we have
\[
\mathbb P \left( S\leqslant \tfrac{t}{2} \,\big|\, \Gamma=\gamma \right)
\geq
\mathbb P\left(J_{\gamma,s,\gamma(s)}^c \text{ for some } s\in[0,\tfrac{t}{2}] \right)
.
\]

Let $\sqrt{t} \leqslant t_1<t_2<\cdots<t_k \leqslant \frac{t}{2}$ be such that $t_{j} \geqslant t_{j-1} + \sqrt{t}$ and $[t_j-\sqrt{t},t_j)$ is a favorable interval for $\gamma$.
Let $\boldsymbol{z}_j = (\gamma(t_j),t_j)$ and write
\begin{align*}
F_t&=\left\{(x,-u), \|x\|_\infty = \lfloor 4\beta u\rfloor\text{ and } 0\leq u <\sqrt{t}\right\}\cup \left\{u=-\sqrt{t}\text{ and } \|x\|_\infty\geqslant 4\beta\sqrt{t}\right\}.
\end{align*}
By definition of favorable interval and of the set $F_t$, the path $\gamma$ cannot make enough jumps to leave $F_t$ so we have $D_\gamma \cap \left(\R^d \times [t_j-\sqrt{t},t_j)\right)\subseteq \boldsymbol{z}_j + F_t$.
If $J_{\boldsymbol{z}_j}:=J_{\gamma,t_j,\gamma(t_j)}$ occurs then $\boldsymbol{z}_j +F_t \leadsto \boldsymbol{z}_j+E_t$.
Since these events depend on $U \cap \left(\R^d \times [t_j-\sqrt{t},t_j)\right)$,
which are disjoint as $j$ goes from 1 to $k$, we have that
\[
\P\left( S \leqslant \tfrac{t}{2} \,\big|\, \Gamma=\gamma \right)
\geqslant
1-\P\left( \boldsymbol{z}_j + F_t \rightsquigarrow \boldsymbol{z}_j + E_t \text{ for all } j \right)
=
1 - \P\left( F_t \rightsquigarrow E_t \right)^k.
\]
This proves the lemma with
\[p_\beta=\sup_{t\ge1} \P(F_t\leadsto E_t),\]
which is less than one as a consequence of exponential decay \cite[Lemma~2.13]{andjelrolla}.
\end{proof}

\begin{proof}
[Proof of Lemma~\ref{lemma:probabreak1}]
The assumption $R_t = e^{O(t)}$ could be waived, but it makes the proof more transparent.
On the event $G^0(B_{R_t+\beta t})$, no open path can start outside $B_{R_t+\beta t}$
and reach $B_{R_t}$ by time $t$, so $X$ must be in this box.
This implies in particular that $X$ itself is good, and by Lemma~\ref{lemma:favorableintervals} the path $\Gamma$ will have at least $\frac{\sqrt{t}}{4}-1$ disjoint favorable time intervals on $[0,\frac{t}{2}]$.

When $\beta$ and $t$ are large enough, Corollary~\ref{cor:goodconditoned} gives
\[
\P \left( G^0(B_{R_t+\beta t}) \,\middle|\big.\, \eta_t\cap B_{R_t}\neq\emptyset \right) \ge 1 - e^{-t}
.
\]
Using Lemma~\ref{L7} we get
\[
\P \left( S\leq \tfrac{t}{2}\,\big|\,\eta_t\cap B_{R_t}\neq\emptyset \right)
\geq
\P \left( S\leq \tfrac{t}{2}, G^0(B_{R_t+\beta t}) \,\big|\,\eta_t\cap B_{R_t}\neq\emptyset \right)
\geq
1-e^{-t}-p_{\beta}^{\frac{\sqrt{t}}{4}-1}
.
\]
We obtain the result by taking $0<c_{\beta}<\frac{\ln(1/p_{\beta})}{4}$.
\end{proof}

\subsection{Expectation of the number of infected sites}
\label{sec:expectation}

We finally prove Proposition~\ref{prop:convexpectyaglom}.
The notation $\E_\nu|\zeta|$ in its statement is a short for $\sum_{\zeta \in \Lambda} |\zeta| \, \nu(\zeta)$, where $\nu$ is the limiting distribution in~(\ref{eq:yaglom}). We first state two crucial facts about this Yaglom limit and then proceed to the proofs.

\begin{lemma}
\label{lemma:finiteexp}
$\E_\nu|\zeta| < \infty$.
\end{lemma}

\begin{lemma}
\label{lemma:uniform}
The family of conditional distributions $\law\left( | \zeta_{r}^{\o} |\,\middle|\,\big. \zeta_{r}^\o \ne \emptyset \right)$ on $\N$ indexed by $r\in[0,\infty)$ is uniformly integrable, and
$
\E\left( | \zeta_{r}^{\o} |\,\middle|\,\big. \zeta_{r}^\o \ne \emptyset \right)
{\rightarrow \E_\nu |\zeta|}
\text{ as }
r\to\infty
.
$
\end{lemma}

\begin{proof}
[Proof of Proposition~\ref{prop:convexpectyaglom}]
Consider the complementary events
\begin{align*}
F_{1a} & = \widehat{G}_Y^S, S\leq t/2,Y\in B_{R_t-\beta t}
,
\\
F_{1b} & = \widehat{G}_Y^S, S\leq t/2,Y \not\in B_{R_t-\beta t}
,
\\
F_{2} & =
S > t/2
\text{ or } 
\neg \widehat{G}_Y^S
,
\end{align*}
ordered from most to least likely, and let $F_1 = F_{1a} \cup F_{1b}$.
Writing $\tilde{\P}$ and $\tilde{\E}$ to represent the conditioning on $\eta_t\cap B_{R_t}\neq \emptyset$, we start by decomposing $\tilde{\E}\left(|\eta_t\cap B_{R_t}| \right)$ as
\begin{equation*}
\tilde{\E}\left(|\eta_t\cap B_{R_t}| \big. \right)
=
\tilde{\E}\left( |\eta_t\cap B_{R_t}| \cdot \1_{F_1} \big. \right)
+
\tilde{\E}\left( |\eta_t\cap B_{R_t}| \cdot \1_{F_2} \big. \right)
.
\end{equation*}
The proof consists in showing that the first expectation converges to $\E_\nu |\zeta|$ and that $\tilde{\P}(F_2) \to 0$ fast enough.
Indeed, using Lemmas~\ref{lemma:breakpointgood} and~\ref{lemma:probabreak1},
\[
\tilde{\E}\left( |\eta_t\cap B_{R_t}| \cdot \1_{F_2} \big. \right)
\leq
|B_{R_t}| \cdot \tilde{\P} (F_2)
\leq
|B_{R_t}| \cdot \left( e^{-c t} + e^{-\sqrt{t}} \right)
\to 0
.
\]
On the other hand,
\[
|\eta_t\cap B_{R_t}| \cdot \1_{F_1} = 
|\eta_t\cap B_{\beta t}^Y | \cdot \1_{F_{1a}} + 
|\eta_t\cap B_{R_t} \cap B_{\beta t}^Y | \cdot \1_{F_{1b}} + 
|\eta_t\cap B_{R_t} \setminus B_{\beta t}^Y | \cdot \1_{F_{1}}
.
\]
By Lemmas~\ref{lemma:steptwo} and~\ref{lemma:probabreak2},
\begin{gather*}
\tilde{\E}\left( |\eta_t\cap B_{R_t} \setminus B_{\beta t}^Y | \cdot \1_{F_{1}} \right)
\leq
{\E}\left( |\eta_t\cap B_{R_t} \setminus B_{\beta t}^Y| \,\middle|\, {F_{1}} \right)
\leq
|B_{R_t}| e^{-\alpha t/2}
\to 0
,
\\
\tilde{\E}\left( |\eta_t\cap B_{R_t} \cap B_{\beta t}^Y | \cdot \1_{F_{1b}} \right)
\leq
|B_{\beta t}| \, \tilde{\P}( F_{1b} )
\leq
C_\beta t^d \frac{\, t^{d+1} \,}{R_t}
\to 0
.
\end{gather*}
Since $\tilde{\P}(F_{1a})\to 1$ as $t\to\infty$, it suffices to show that
\begin{gather*}
{\E}\left( |\eta_t \cap B_{\beta t}^Y | \,\middle|\, {F_{1a}} \right)
\to
\E_\nu|\zeta|
.
\end{gather*}
Using Lemma~\ref{lemma:stepone} and translation invariance
\begin{align*}
{\E}\left( |\eta_t \cap B_{\beta t}^Y | \,\middle|\, {F_{1a}} \right)
&=
\int_{y,s} \E\left(|\eta_t\cap B_{\beta t}^y|\,\big|\,\widehat{G}_y^s,Y=y, S=s \right) \dd\P(Y=y,S=s | F_{1a})
\\
&=
\int_{[0,\frac{t}{2}]} \E\left(|\eta_{t-s}^{\o}|\,\big|\, \eta_{t-s}^{\o}\ne\emptyset, G_{\o}^0 \right) \dd\P(S=s | F_{1a})
.
\end{align*}
On the other hand, Corollary~\ref{cor:goodconditoned} gives $\P(G^0_\o \,|\, \eta^\o_{r}\ne\emptyset )\to 1$ uniformly over $r\in[0,{t}]$.
Combining this with Lemma~\ref{lemma:uniform}, we get
\begin{align*}
\E\left(|\eta_{r}^{\o}|\,\big|\, \eta_{r}^{\o}\ne\emptyset, G_{\o}^0 \right)
\to
\E_\nu |\zeta|
\quad
\text{ as }
\quad
t\to\infty
,
\end{align*}
uniformly over $r\in[\frac{t}{2},t]$,
concluding the proof.
\end{proof}

\begin{proof}
[Proof of Lemma~\ref{lemma:finiteexp}]
We discuss how results of~\cite{swart,sturmswart}, translated to our setting, say that~$\nu$ is the unique QSD satisfying
\(
\E_\nu |\zeta|<\infty.
\)
Recall that $\Lambda$ is the quotient space resulting from translation equivalence and $(\zeta_t)$ is the contact process modulo translation taking values in $\Lambda\cup\{\emptyset\}$.
Let $\mathcal{Z}=\{A\subset \Z^d:A\text{ is non empty and finite}\}$.
Consider the kernel $P_t$ acting on (possibly infinite) measures $\mu$ on $\mathcal{Z}$ by
\[
(\mu P_t) ( \D ) = \sum_{A\in\mathcal{Z}} \mu(A) \, \P (\eta^A_t \in \D)
,
\quad
\D\subset \mathcal{Z}
.
\]
Proposition~1.4 in~\cite{swart} states that there exists a unique (modulo multiples) translation-invariant locally-finite eigenmeasure $\mu$. That is, there is a unique $\mu$ satisfying:

(i) $\mu(A+x)=\mu(A)$
for all $A$ and $x$;
(ii) $\sum_{\mathcal{Z}} \1_{\o \in A} \, \mu(A) < \infty$;
(iii) $\mu P_t=e^{rt}\mu$
for $t>0$.

We can thus consider the probability measure $\tilde{\mu}$ on $\Lambda$ defined by $\tilde{\mu}(\langle A \rangle)=c\mu(A)$ for each $A\in\mathcal{Z}$. By (i), $\tilde{\mu}$ is well defined. By (iii), the measure $\tilde{\mu}$ satisfies $\tilde{\mu} P_t=e^{-\alpha t}\tilde{\mu}$, so it is a quasi-stationary distribution for the contact process modulo translations.
By (ii),
\begin{equation*}
\E_{\tilde{\mu}} |\zeta| = \sum_{\zeta \in \Lambda} |\zeta| \cdot {\tilde{\mu}}(\zeta)
=
\sum_{\zeta \in \Lambda} \sum_{A\in\mathcal{Z}}
\1_{\langle A \rangle = \zeta}
\,
\1_{\o \in A}
\cdot
{\tilde{\mu}}(\zeta)
=
c
\sum_{A\in\mathcal{Z}}
\1_{\o \in A}
\cdot
\mu(A)
<
\infty
\end{equation*}
Finally, \cite[Theorem~2.12]{sturmswart} states that $\tilde{\mu}$ satisfies the Yaglom limit, so $\tilde{\mu}=\nu$.
\end{proof}

\begin{proof}
[Proof of Lemma~\ref{lemma:uniform}]
Let us identify measures on $\Lambda$ with row vectors, and real functions on $\Lambda$ with column vectors, both indexed by $\Lambda$. Consider the sub-Markovian semigroup of matrices $(P^t)_{t \ge 0}$ defined by $P^t(A,B)=\P(\zeta^A_t=B)$ for $A,B\in\Lambda$.

It follows from \cite[Section~3]{andjelrolla} that $P^t$ is $\alpha$-positive and there exist a probability $\nu$ and positive function $h$ such that
\(
e^{\alpha t}P^t \to h \nu
\text{ and }
e^{\alpha t}P^t \boldsymbol{1} \to h
\)
as $t\to\infty$.

In the notation of~\cite{nummelin}, for a nonnegative function $g$ and a signed measure $\mu$, define $\|\mu\|_g = \max\{\mu^+ g , \mu^- g\}$.
Writing $\pi=\nu$ and
taking
\(
g(\zeta)=|\zeta|,
\)
from Lemma~\ref{lemma:finiteexp} we have $\pi g < \infty$, so we can apply \cite[Theorem~2]{nummelin}, obtaining
\[
\left\| e^{\alpha t} P^t(A,\cdot) - h(A) \pi(\cdot) \right\|_g \to 0
.
\]
Now, as a consequence of $\alpha$-positivity with summable $\nu$,
the denominator in
\[
\P \left( \zeta_t^A = \,\cdot\,\, \middle|\, \tau>t \right)
=
\frac
{e^{\alpha t}P^t(A,\,\cdot\,)}
{e^{\alpha t}P^t(A,\Lambda)}
,
\]
converges to $h(A)$, see \cite[Theorem~3.1]{andjelrolla}.
Therefore
\[
\left\| \P ( \zeta_t^A = \,\cdot\,\, |\, \tau>t ) - \nu \right\|_g \to 0
,
\]
which implies uniform integrability as stated.
\end{proof}

\section{Proof of the scaling limit}
\label{sec:scaling}

In this section we prove Theorem~\ref{thm:convergence}. The proof relies on Propositions~\ref{prop:cvtonu} and~\ref{prop:convexpectyaglom}, but is otherwise independent of the previous section.
In the sequel we state and discuss three independent lemmas, and use them to prove the theorem.
Proving the lemmas is postponed to the end of the section.

We start with the fact that connected components in $\mathcal{C}_t$ can indeed be represented by some $(x,\zeta) \in \Z^d \times \Lambda$.
\begin{lemma}
\label{lemma:nonperco}
If $R_t\ll e^{\alpha t/d}$, for large $t$ there are a.s.\ no infinite components in $\mathcal{C}_t$.
\end{lemma}
This lemma uses the definition of good points and the crude estimate
$\P(\eta_t \cap B_{R_t}\neq \emptyset) \ll 1$.
Using Proposition~\ref{prop:convexpectyaglom}, we obtain the correct asymptotics, given by the following lemma.
\begin{lemma}
\label{lemma:prefactor}
If $t^{2d+1}\ll R_t\ll e^{\sqrt[3]{t}}$, then
for the constant $\rho$ given by~(\ref{eq:rho}) we have
\[
{\P\left(\big. \eta_t\cap B_{R_t} \neq \emptyset\right)}
\sim
\rho
\,
e^{-\alpha t}
\,
{|B_{R_t}|}
.
\]
\end{lemma}
So the density of infected boxes agrees with the density of the scaling limit stated in Theorem~\ref{thm:convergence}. Moreover, knowing which boxes are infected we have the correct distribution inside each of them, thanks to Proposition~\ref{prop:cvtonu}.

The missing ingredient is independence between boxes.
In the sequel we give a precise meaning to the idea that the configuration $\eta_t$ is almost independent across different boxes.
Namely, we will show that, for large $t$, the configuration $\eta_t$ can be approximated by a collection of independent patches within mesoscopic boxes.

From now on we let
\[
K>0
\quad
\text{ and }
\quad
\Psi=[-K,K]^d \subset \R^d
\]
be fixed, and assume that $1 \ll R_t \ll e^{\alpha t/d}$.
Let
$$\mathcal{R}_t=e^{-\frac{\alpha t}{d}} \, \Z^d$$
be the lattice (gray grid in Figure~\ref{micro}) where the rescaled configuration
\[
{\oeta}_t = e^{-\frac{\alpha t}{d}} \eta_t \subset {\mathcal{R}_t}
\]
lives (pink points in Figure~\ref{micro}).
The space $\R^d$ will be partitioned into mesoscopic boxes.
More precisely, consider the intermediate lattice
\[
\mathcal{R}'_t = ( 2R_t + 1 ) \mathcal{R}_t
\]
associated to the centers of such boxes (blue points in Figure~\ref{micro}).
Label the points in $\mathcal{R}'_t$ as $(y_t^{(j)},j\geq1)$.
Finally, the box associated to $y_t^{(j)}$ is denoted by
\begin{align*}
B_{R_t}^{(j)}&=\mathcal{B} \big( y_t^{(j)}, R_t e^{-\frac{\alpha t}{d}} \big)
\cap
\mathcal{R}_t
\end{align*}
which corresponds to a discrete ball of radius $R_t$ in $\mathcal{R}_t$. There are
\begin{equation}
\label{eq:boxcount}
n_t(\Psi) \sim \frac{K^d}{\, e^{-\alpha t} R_t^d\, }
\end{equation}
such boxes intersecting $\Psi$, and we assume they are labeled $B_{R_t}^{(1)},\dots,B_{R_t}^{(n_t)}$.

\begin{lemma}
\label{lemma:indepbetweenboxes}
If $R_t \gg t$, then, as $t \to \infty$,
\[\law\left(\oeta_t \cap B_{R_t}^{(j)}\cap \Psi,j\geq 1\right)\approx\bigotimes_{j\ge 1} \law \left(\oeta_t \cap B_{R_t}^{(j)}\cap \Psi \right),\]
where the right hand side is composed by independent copies of the contact process.\end{lemma}

We are now ready to prove the scaling limit.

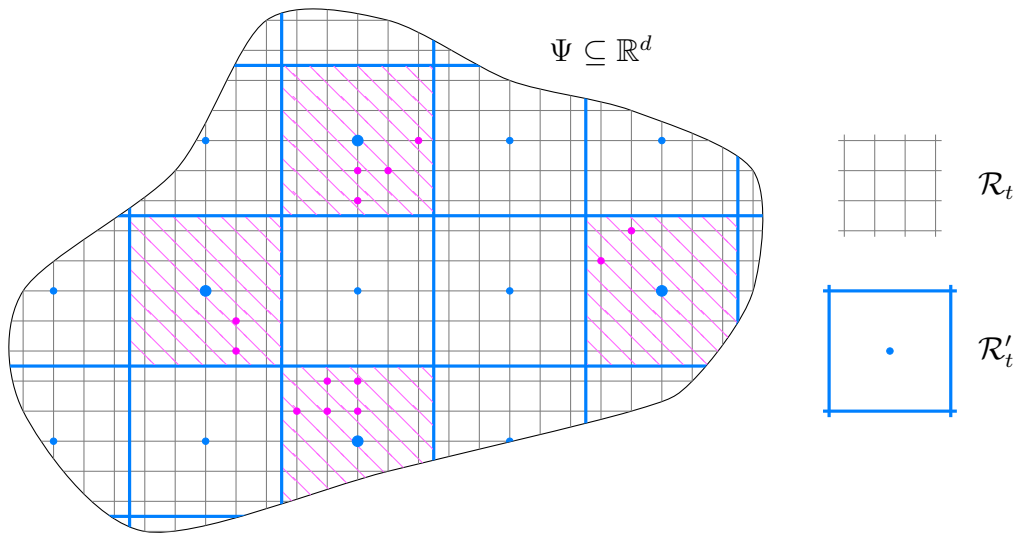
\begin{figure}[!hb]
{\centering\begin{tikzpicture}[scale=0.4,rotate=90]
\foreach \x in {10, ...,13}
  \draw[color=gray] (\x,-2.8) -- (\x,-6.2);
  \foreach \x in {-3, ...,-6}
  \draw[color=gray] (9.8,\x) -- (13.2,\x);
  \foreach \x in {4, 8}
  \draw[color=myblue,very thick] (\x,-2.3) -- (\x,-6.7);
    \foreach \x in {-2.5, -6.5}
  \draw[color=myblue,very thick] (3.8,\x) -- (8.2,\x);
  \draw[color=myblue,fill] (6, -4.5) circle[radius=3pt];
  \draw (11.5,-8) node{$\mathcal{R}_t$};
   \draw (6, -8) node{$\mathcal{R}'_t$};
\draw (16,5) node{$\Psi\subseteq \R^d$};
\draw[clip] plot [smooth cycle] coordinates {%
            ( 5,  2)  ( 8,  0) ( 12,  0)  %
            (14,  4)  (15,  8)  (17, 12)  
            (17, 16)  (12, 19)  ( 8, 24)  %
            ( 4, 24)  ( 0, 20)  ( 2, 12)  %
            ( 4, 4) };
\foreach \x in {0, ..., 30}
  \draw[color=gray] (-0.5, \x) -- (17.5, \x);
\foreach \x in {0, ..., 17}
  \draw[color=gray] (\x, -0.5) -- (\x, 30.5);
\foreach \x in {0, 5, 10, 15, 20, 25, 30 }
  \draw[color=myblue,very thick] (-0.5, \x.5) -- (17.5, \x.5);
\foreach \x in {0, 5, 10, 15, 20}
  \draw[color=myblue,very thick] (\x.5, -0.5) -- (\x.5, 30.5);
\foreach \x in {0, 5, 10, 15, 20}
  \foreach \y in {0, 5, 10, 15, 20, 25, 30 }
     \draw[color=myblue,fill] (\x+3, \y+3) circle[radius=3pt];
\def\hashedcase(#1,#2){%
\fill[pattern=north west spaced lines, pattern color=mypink!60] (#1+0.55, #2+0.55) rectangle (#1+5.45, #2+5.45);%
\draw[color=myblue,fill] (#1+3, #2+3) circle[radius=5pt];}
\hashedcase(5, 0);
\hashedcase(0, 10);
\hashedcase(5, 15);
\hashedcase(10, 10);
\draw[color=mypink,fill] ( 9,  5) circle[radius=3pt];
\draw[color=mypink,fill] (10,  4) circle[radius=3pt];
\draw[color=mypink,fill] ( 4, 13) circle[radius=3pt];
\draw[color=mypink,fill] ( 4, 14) circle[radius=3pt];
\draw[color=mypink,fill] ( 4, 15) circle[radius=3pt];
\draw[color=mypink,fill] ( 5, 13) circle[radius=3pt];
\draw[color=mypink,fill] ( 5, 14) circle[radius=3pt];
\draw[color=mypink,fill] ( 7, 17) circle[radius=3pt];
\draw[color=mypink,fill] ( 6, 17) circle[radius=3pt];
\draw[color=mypink,fill] (13, 11) circle[radius=3pt];
\draw[color=mypink,fill] (12, 12) circle[radius=3pt];
\draw[color=mypink,fill] (11, 13) circle[radius=3pt];
\draw[color=mypink,fill] (12, 13) circle[radius=3pt];

\end{tikzpicture}\par}
\caption{(Color online) Microscopic and mesoscopic lattices}
\label{micro}
\end{figure}

\begin{proof}
[Proof of Theorem~\ref{thm:convergence}]
Consider the random measure
\[
M_t=\sum_{(x,\zeta)\in\C_t} \delta_{(e^{-\frac{\alpha}{d} t }x,\zeta),}
\]
where $(x,\zeta)$ represents each connected component of $\eta_t$ as defined at the introduction.
Let $M$ be a Poisson random measure on $\R^d\times \Lambda$ with intensity $\rho \dd x\times \nu$.

We want to show that $M_t$ converges in distribution to $M$.
We first prove convergence assuming that $$t^{2d+1} \ll R_t \ll e^{\sqrt[3]{t}},$$ and then extend it to any other $1 \ll R_t \ll e^{\alpha t/d}$.
Let $\teta_t$ be obtained by patching a collection of independent copies of $\eta_t$ on boxes of radius $R_t$.
We denote $\tilde{\C}_t$ and $\tM_t$ to indicate the use of $\teta_t$ instead of $\eta_t$.
By Lemma~\ref{lemma:indepbetweenboxes}, it suffices to show that $\tM_t$ converges in distribution to $M$.
We will write $$\oeta_t = e^{-\alpha t/d} \teta_t$$ for the rescaled configuration.

Now let $\Xi\subset \Psi$ be a compact rectangle and let $\D \subset \Lambda$ be a finite set of configurations.
By \cite[Proposition~3.22]{resnick}, it suffices to show that
\begin{equation}
\label{eq:pppempty}
\P \left(\tM_t(\Xi \times \D)=0 \big.\right)
\to
\P \left( M(\Xi \times \D)=0 \big.\right)
=
e^{-\rho|\Xi|\nu(\D)}
\end{equation}
and
\begin{equation}
\label{eq:pppmean}
\E \left( \tM_t(\Xi \times \D) \big.\right) \to
\E \left( M(\Xi \times \D) \big.\right)
=
\rho |\Xi| \nu(\D)
\end{equation}
as $t \to \infty$.

Write $p_t=\P(\oeta_t\cap B_{R_t}^{(1)}\neq\emptyset)$ 
By Lemma~\ref{lemma:prefactor} we have $p_t \sim \rho |B_{R_t}| e^{-\alpha t}$.
Moreover, using Proposition~\ref{prop:cvtonu} we get
\begin{equation}
\label{eq:boxconfiguration}
\P \left( \langle \oeta_t\cap B_{R_t}^{(1)} \rangle \in \D \big.\right) \sim \rho |B_{R_t}| e^{-\alpha t} \nu(\D).
\end{equation}

Consider $\tilde{\C}_t^\#$ obtained from $\teta$ similarly to $\tilde{\C}_t$, except that we only connect infected sites lying within the same box $B_{R_t}^{(j)}$. Denote the corresponding random measure by $\tM_t^\#$.
Let $d_\Psi(\tM_t,\tM_t^\#)$ count the number of point masses in $\Psi \times \Lambda$ that are present in one of the random measures and not in the other.
Then
\begin{multline*}
\P\left( \tM_t \ne \tM_t^\# \text{ in } \Psi \right)
=
\P \left( d_\Psi(\tM_t,\tM_t^\#) \geq 1 \right)
\leq
\E \left[ d_\Psi(\tM_t,\tM_t^\#) \right]
\leq
\\
\leq
2 \sum_{j=1}^{n_t(\Psi)} 
\P\left(\oeta_t\cap B_{R_t}^{(j)}\neq\emptyset \text{ and } \oeta_t\cap B_{R_t}^{(i)}\neq\emptyset \text{ for some } B_{R_t}^{(i)}\text{ neighbor of }B_{R_t}^{(j)} \right)
=
\\
=
2 \, n_t \, (3^d-1) \, p_t^2
\lesssim
2 \cdot 3^d \, \rho \, |\Psi| \cdot p_t \to 0,
\end{multline*}
so we can consider $\tM_t^\#$ instead of $\tM_t^\#$ in~(\ref{eq:pppempty}) and~(\ref{eq:pppmean}).
Label the boxes $B_{R_t}^{(j)}$ so that the first $n_t(\Xi)$ intersect $\Xi$ and the first $n_t(\Psi)$ intersect $\Psi$.
Using independence and~(\ref{eq:boxconfiguration}),
\begin{multline*}
\P \left(\tM_t^\#(\Xi \times \D)=0 \big.\right)
=
\P \left( \langle \oeta_t\cap B_{R_t}^{(j)} \rangle \notin \D ,\ i=1,\dots,n_t(\Xi) \big.\right)
=
\\
=
\left[ 1 - \P \left( \langle \oeta_t\cap B_{R_t}^{(1)} \rangle \in \D \big.\right) \right]^{n_t(\Xi)}
\approx
e^{- \rho |B_{R_t}| e^{-\alpha t} \nu(\D) n_t(\Xi)}
\approx
e^{- \rho |\Xi| \nu(\D)}
,
\end{multline*}
and similarly,
\begin{equation*}
\E \left(\tM_t^\#(\Xi \times \D) \big.\right)
=
\sum_{i=1}^{n_t(\Xi)}
\P \left( \langle \oeta_t\cap B_{R_t}^{(j)} \rangle \in \D \big.\right)
\approx
\rho |B_{R_t}| e^{-\alpha t} \nu(\D) n_t(\Xi)
\approx
\rho |\Xi| \nu(\D)
.
\end{equation*}
This proves the required convergence in distribution for $t^{2d+1} \ll R_t \ll e^{\sqrt[3]{t}}$.

To conclude the proof, suppose that $1 \ll R_t \ll e^{\alpha t/d}$.
For $R_t' = t^{2d+2}$, we already proved that $M_t' \stackrel{d}{\to} M$.
We will show that $M_t \cdot \1_\Psi = M_t' \cdot \1_\Psi$ with high probability as $t\to\infty$.

For convenience, take $w_t$ and $W_t$ such that $R_t , R_t' \in [w_t , W_t]$ for $t\ge 0$, and such that $w_t\to\infty$ and $\epsilon_t = \frac{W_t}{ e^{\alpha t/d}} \to 0$. Notice that the occurrence of $M_t \cdot \1_\Psi \ne M_t'\cdot \1_\Psi$ implies that $\eta_t$ contains a pair of infected sites in the box $e^{\alpha t/d}\Psi$ whose distance is in $[w_t,W_t]$, or an infected site at distance less than $W_t$ from the outside of the box.
Indeed, if every pair of infected sites are distant less than $w_t$ or more than $W_t$, then either they are connected for being closer than both $R_t$ and $R_t'$, or they are not connected for being farther than both $R_t$ and $R_t'$, and thus $M_t \cdot \1_\Psi = M_t' \cdot \1_\Psi$.
Now, unless $\eta_t$ contains infected sites $W_t$-close to $\partial e^{\alpha t/d}\Psi$ (which has probability bounded by $2d (2K)^{d-1} \epsilon_t$, vanishing as $t\to\infty$), the occurrence of a pair of infected sites whose distance is in $[w_t,W_t]$ implies that $M_t' \cdot \1_{\Psi}$ either contains a cluster $(x,\zeta)$ with $\diam(\zeta)\geq w_t$, or it contains a pair of clusters $(x^1,\zeta^1)$ and $(x^2,\zeta^2)$ with $\| x^1 - x^2 \|_\infty \leq 3 \epsilon_t$.
But the probability that $M_t'$ contains such clusters approximates the probability that the Poisson process $M$ contains such clusters.
The latter probability, in turn, is bounded by $\rho \, |\Psi| \, \nu\{\zeta:\diam{\zeta}\geq w_t\} + \rho^2 (2\epsilon_t)^d |\Psi|$, which is arbitrarily small when $t$ is large enough.
\end{proof}

\begin{proof}
[Proof of Lemma~\ref{lemma:prefactor}]
Using the fact that
\[
\P\left( \o \in \eta_t \right) \sim h(\{\o\}) \, e^{-\alpha t},
\]
we compute the probability to have infected points in the box $B_{R_t}$ by
\begin{align*}
\P( \eta_t\cap B_{R_t}\neq \emptyset)
&=
\frac{\E\left(|\eta_t\cap B_{R_t}| \right)}{\E\left(|\eta_t\cap B_{R_t}|\,\big|\, \eta_t\cap B_{R_t}\neq \emptyset\right)}
=\frac{|B_{R_t}|\P(\o\in \eta_t)}{\E\left(|\eta_t\cap B_{R_t}|\,\big|\, \eta_t\cap B_{R_t}\neq \emptyset\right)}.
\end{align*}
Thus, using Proposition~\ref{prop:convexpectyaglom} and~(\ref{eq:hprefactor}), we get
\begin{equation*}
\frac{\P( \eta_t\cap B_{R_t}\neq \emptyset)}{e^{-\alpha t} \, |B_{R_t}| } = \frac{e^{\alpha t} \, \P(\o\in \eta_t)}{\E\left(|\eta_t\cap B_{R_t}|\,\big|\, \eta_t\cap B_{R_t}\neq \emptyset\right)}\xrightarrow[t\to\infty]{}\frac{h(\{\o\})}{\E_\nu |\zeta|}.
\qedhere
\end{equation*}
\end{proof}

\begin{proof}
[Proof of Lemma~\ref{lemma:indepbetweenboxes}]
Let
$$\displaystyle\Psi^+ = [-K-\beta t e^{-\alpha t/d},K+\beta t e^{-\alpha t/d}]^d.$$
Let $t>0$ and $(B_{R_t}^{(j)})_{j\in\{1,\ldots,n_t\}}$ be a subfamily of balls covering $\Psi$. The idea is the following: with high probability, we can suppose that the boundaries of $\Psi^+$ and the boxes ${B}_{R_t-\beta t}^{(j)}$ are good and that only the points in ${B}_{R_t-\beta t}^{(j)}$ (and outside of $\Psi^+$) survive, in which case $\oeta_t\cap {B}_{R_t}^{(j)}$ depends only on ${B}_{R_t}^{(j)}\times[0,t]$ which gives us the independence.

Denote
\[
B_{r}^{(j)} = \mathcal{B} ( y_t^{(j)}, r e^{-\frac{\alpha t}{d}} ) \cap \mathcal{R}_t
\quad
\text{ and }
\quad
D_{r}^{(j)} = B_{r}^{(j)} \setminus B_{r-1}^{(j)}
.
\]
For brevity write
$\mathring{B}_t^j = B_{R_t-\beta t}^{(j)}$, $\mathring{D}_t^j = D_{R_t-\beta t}^{(j)}$ and $\mathring{B}_t=\cup_{1\leq j\leq n_t}\mathring{B}_t^j $.
\begin{align}
\nonumber
&
\limsup_{t\to\infty} \bigg\| \law\left(\oeta_t \cap B_{R_t}^{(j)}\cap\Psi,1\le j\le n_t\right)-\bigotimes_{j= 1}^{n_t} \law \bigg(\oeta_t \cap B_{R_t}^{(j)}\cap \Psi \bigg)\bigg\|
\\
\nonumber
& \phantom{a}
\leq \limsup_{t\to\infty} \bigg\| \law\left(\oeta_t \cap B_{R_t}^{(j)},1\le j\le n_t\,\big|\,G^0\left(\partial\Psi^+\right)\right)-\bigotimes_{j= 1}^{n_t} \law \bigg(\oeta_t \cap B_{R_t}^{(j)} \,\big|\,G^0\left(\partial\Psi^+\right)\bigg)\bigg\|
\\
\label{eq:zerolimsup}
& \qquad \qquad
+ \limsup_{t\to\infty} 2 \P\left( \neg G^0\left(\partial\Psi^+\right) \right)
\\
\nonumber
& \phantom{a}
\leq \limsup_{t\to\infty} \bigg\| \law\left(\oeta_t^{\mathring{B}_t} \cap B_{R_t}^{(j)},1\le j\le n_t\right) - \bigotimes_{j=1}^{n_t} \law \left(\oeta_t^{\mathring{B}_t^j} \cap B_{R_t}^{(j)}\,\right)\bigg\| +
\\
\label{eq:firstlimsup}
& \qquad\qquad
+ \limsup_{t\to\infty} 2 \sum_{j=1}^{n_t} \big|B_{R_t}^{(j)} \setminus {\mathring{B}_t^j} \big| e^{-\alpha t}
\\
\nonumber
& \phantom{a}
\leq \limsup_{t\to\infty} \left\| \law\left(\oeta_t^{\mathring{B}_t} \cap B_{R_t}^{(j)},1\le j\le n_t \,\big|\, G^0(\cup_{j=1}^{n_t} \mathring{D}_{t}^{j})\right) - \bigotimes_{j=1}^{n_t} \law \left(\oeta_t^{\mathring{B}_t^j} \cap B_{R_t}^{(j)}\,\right)\right\|
\\
\label{eq:secondlimsup}
& \qquad \qquad
+ \limsup_{t\to\infty} n_t\cdot \P\left( \neg G^0(\mathring{D}_{t}^{1}) \right)
\\
\nonumber
& \phantom{a}
= \limsup_{t\to\infty} \left\| \bigotimes_{j=1}^{n_t}\law\left(\oeta_t^{\mathring{B}_t^j} \cap B_{R_t}^{(j)}\,\Big|\, G^0(\mathring{D}_{t}^{j})\right) - \bigotimes_{j= 1}^{n_t} \law \left(\oeta_t^{\mathring{B}_t^j} \cap B_{R_t}^{(j)}\,\right)\right\|
\\
\nonumber
& \phantom{a}
= \limsup_{t\to\infty} {n_t}\cdot \left\| \law\left(\oeta_t^{\mathring{B}_t^1} \cap B_{R_t}^{(1)} \,\Big|\, G^0(\mathring{D}_{t}^{1})\right) - \law \left(\oeta_t^{\mathring{B}_t^1} \cap B_{R_t}^{(1)}\,\right)\right\|
\\
\label{eq:thirdlimsup}
& \phantom{a}
\le \limsup_{t\to\infty} n_t\cdot \P\left( \neg G^0(\mathring{D}_{t}^{1}) \right)
= 0
.
\end{align}
The terms~(\ref{eq:zerolimsup}),~(\ref{eq:secondlimsup}) and~(\ref{eq:thirdlimsup}) vanish for $\beta$ large enough by Corollary~\ref{cor:goodwithout}.
The term in~(\ref{eq:firstlimsup}) is bounded by $2 n_t \cdot C \beta t R_t^{d-1} \cdot e^{-\alpha t}$, which by~(\ref{eq:boxcount}) vanishes since $R_t \gg t$.

In the second inequality we excluded some sites from the initial configuration and used~(\ref{eq:usefulestimates}).
The first equality uses the fact that, conditioning on the event of all good borders, the processes in each box are independent and stay within the respective boxes.
\end{proof}

\begin{proof}
[Proof of Lemma~\ref{lemma:nonperco}]
We can assume that $R_t \geq \beta t$, since
decreasing $R_t$ only removes connections.
For fixed $t$, consider the partitioning of $\Z^d$ into boxes given by
\[
B^{x} = (2R_t + 1)x + B^{}_{R_t}, \quad x\in \Z^d
.
\]
To control interaction between distant boxes, we will consider the boundary of an enlarged box, also indexed by $x\in \Z^d$, namely
\[
D^x = (2R_t + 1)x + D^{}_{R_t+\beta t}
\]
We say that two boxes $B^x$ and $B^y$ are \emph{neighbors} if $\|x-y\|_{\infty} = 1$, so each box has $3^d-1$ neighbors.
If the event $G^0(D^x)$ occurs and $\eta_t \cap B^x = \emptyset$, we say that the box $B^x$ is \emph{good}, otherwise it is \emph{bad}. Since paths starting from farther boxes need to cross $D^x\times[0,t]$ in order to enter $B^x\times[0,t]$, the event that a box is good is determined by the graphical construction on $B^x\times[0,t]$ and neighboring boxes.
Now, for $t$ and $\beta$ large enough, we have
\begin{align*}
\P \left( B^x \text{ is bad} \right) &
\leq \P\left(B_{R_t+\beta t}\leadsto L_t\right)
+
\P \left( \neg G^0(D_{R_t+\beta t}) \right)
\leq |B_{R_t+\beta t}|e^{-\alpha t} + e^{-\rho t},
\end{align*}
which vanishes for large $t$ since $R_t \ll e^{\alpha t/d}$. So the process of good and bad boxes is a 2-dependent percolation field with low density.
Therefore, for large enough $t$ there is a.s.\ no infinite path of neighboring bad boxes.
To conclude the proof, notice that an infinite component in $\mathcal{C}_t$ would require a sequence of infected sites at time $t$, each one $R_t$-close to their predecessor, and this in turn would imply the existence of an infinite sequence of neighboring bad boxes.
\end{proof}

\section*{Acknowledgments}

This research started during thematic trimester ``Disordered Systems, Random Spatial Processes and Some Applications'' at the Henri Poincaré Institute. LR thanks the IHP and the Fondation Sciences Mathématiques de Paris for generous support.
This project supported by grants
PIP 11220130100521CO,
PICT-2015-3154, PICT-2013-2137, PICT-2012-2744,
Conicet-45955 and MinCyT-BR-13/14.

\bibliographystyle{bib/leo}
\bibliography{bib/biblio}

\end{document}